\documentclass{amsart}
\usepackage{amsmath,amscd,xypic,amssymb,combelow,color,enumitem,graphicx,float}
\usepackage{comment}

\usepackage[normalem]{ulem}
\usepackage[dvipsnames]{xcolor}
\makeatletter
\def\squiggly{\bgroup \markoverwith{\textcolor{black}{\lower3.5\p@\hbox{\sixly \char58}}}\ULon}

\usepackage{xr-hyper}
\usepackage[
pdftex,
bookmarks=false,
colorlinks=true,
debug=true,
pdfnewwindow=true]{hyperref}

\makeatletter
  \@addtoreset{equation}{section}
\makeatother

\newtheorem{theorem}[subsection]{Theorem}

\newtheorem{proposition}[subsection]{Proposition}
\newtheorem{lemma}[subsection]{Lemma}

\newtheorem{conjecture}[subsection]{Conjecture}
\newtheorem{definition}[subsection]{Definition}

\theoremstyle{remark}
\newtheorem{claim}[subsection]{Claim}

\newtheorem{remark}[subsection]{Remark}

\def\fb{{\mathfrak{b}}}
\def\fg{{\mathfrak{g}}}
\def\fh{{\mathfrak{h}}}

\def\fn{{\mathfrak{n}}}

\def\BC{{\mathbb{C}}}

\def\BN{{\mathbb{N}}}
\def\BF{{\mathbb{F}}}

\def\BR{{\mathbb{R}}}
\def\BQ{{\mathbb{Q}}}
\def\BZ{{\mathbb{Z}}}

\def\CB{{\mathcal{B}}}

\def\CS{{\mathcal{S}}}

\def\CV{{\mathcal{V}}}

\def\ph{\varphi}

\def\tt{\tilde{t}}

\def\sym{\textrm{Sym}}

\def\uu{U_q(\fg)}
\def\uup{U_q(\fn^+)}

\def\uuo{U_q(\fh)}
\def\uum{U_q(\fn^-)}

\def\UU{U_q(L\fg)}
\def\UUp{U_q(L\fn^+)}

\def\UUo{U_q(L\fh)}
\def\UUm{U_q(L\fn^-)}

\def\wUUm{\widehat{U}_q(L\fn^-)}

\def\bk{{\mathbf{k}}}
\def\bl{{\mathbf{l}}}

\def\bs{\boldsymbol{\varsigma}}

\def\nn{{\mathbb{N}^I}}
\def\zz{{\mathbb{Z}^I}}

\def\bv{{\boldsymbol{v}}}
\def\bw{{\boldsymbol{w}}}

\def\b0{{\boldsymbol{0}}}

\newcommand\iso{\,\vphantom{j^{X^2}}\smash{\overset{\sim}{\vphantom{\rule{0pt}{0.20em}}\smash{\longrightarrow}}}\,}

\def\wI{\widehat{I}}

\def\wDelta{\widehat{\Delta}}

\def\spec{\text{spec}}
\def\espec{\emph{spec}}

\def\oij{\overrightarrow{ij}}
\def\oji{\overrightarrow{ji}}

\def\sX{X}
\def\Hall{\mathcal{H}}
\def\height{\mathrm{ht}}
\def\tf{\widetilde{f}}
\def\tspec{\widetilde{\spec}}
\def\tespec{\widetilde{\espec}}
\def\tt{v}

\def\UUext{U^{\mathrm{ext}}_q(L\fg)}
\def\UUaff{U_q(\widehat{\fg})}
\def\wUUaff{\widehat{U}_q(\widehat{\fg})}
\def\UUaffp{U_q(\widehat{\fn}^+)}
\def\UUaffm{U_q(\widehat{\fn}^-)}
\def\UUaffpm{U_q(\widehat{\fn}^\pm)}
\def\UUaffh{U_q(\widehat{\fh})}
\def\UUaffg{U_q(\widehat{\fb}^+)}
\def\UUaffl{U_q(\widehat{\fb}^-)}
\def\UUaffgl{U_q(\widehat{\fb}^\pm)}

\def\wDelta{\widehat{\Delta}}
\def\wI{\widehat{I}}

\def\bph{\bar{\ph}}

\begin{document}

\title[Fusion and specialization for type ADE shuffle algebras]
      {\Large{\textbf{Fusion and specialization \\ for type ADE shuffle algebras}}}

\author[Andrei Negu\cb t]{Andrei Negu\cb t}
\address{École Polytechnique Fédérale de Lausanne (EPFL), Lausanne, Switzerland}
\address{Simion Stoilow Institute of Mathematics (IMAR), Bucharest, Romania}
\email{andrei.negut@gmail.com}

\author[Alexander Tsymbaliuk]{Alexander Tsymbaliuk}
\address{Purdue University, Department of Mathematics, West Lafayette, IN, USA}
\email{sashikts@gmail.com}

\maketitle

\begin{abstract}
Root vectors in quantum groups (of finite type) generalize to fused currents in quantum loop groups (\cite{DK}). In the present paper, we construct fused currents as duals to specialization maps of the corresponding shuffle algebras (\cite{E1,E2,FO}) in types ADE, an approach which has potential for generalization to arbitrary Kac--Moody types. Both root vectors and fused currents depend on a convex order of the positive roots, and the choice we make in the present paper is that of the Auslander--Reiten order (\cite{Ri}) corresponding to an orientation of the type ADE Dynkin diagram.
\end{abstract}

\bigskip


\section{Introduction}


\medskip

\subsection{Lie algebras of finite type}
\label{sub:intro}

Consider a finite-dimensional simple Lie algebra $\fg$ over $\BC$, which admits a triangular decomposition
$$
  \fg = \fn^+ \oplus \fh \oplus \fn^-
$$
into nilpotent parts $\fn^\pm$ and a Cartan part $\fh$. We have canonical decompositions
$$
  \fn^\pm = \bigoplus_{\alpha \in \Delta^+} \BC \cdot \sX_{\pm \alpha}
$$
($\Delta^+$ denotes a henceforth fixed set of positive roots of the Lie algebra $\fg$). The reason the one-dimensional direct summands are canonical is that they are determined by their commutation relations with the Cartan subalgebra $\fh$, in the sense that
$$
  [h,\sX_{\pm \alpha}] = \pm \alpha(h) \cdot \sX_{\pm \alpha}, \qquad \forall\, h\in \fh
$$
However, picking a basis of the one-dimensional summands is non-canonical, because the $\sX_{\pm \alpha}$'s are only determined up to constant multiples. Numerous interesting choices exist, notably that of a Chevalley basis, for which\footnote{For any ring $R$, we shall use $R^\times$ to denote the set of all nonzero elements in $R$.}
\begin{equation}
\label{eqn:chevalley}
  [\sX_{\pm \alpha}, \sX_{\pm \beta}] \in \BZ^\times \cdot \sX_{\pm(\alpha+\beta)}, \qquad \text{whenever } \alpha,\beta,\alpha+\beta\in \Delta^+
\end{equation}
This formula can be used to successively construct (up to constant multiples) the basis vectors $\sX_{\pm \alpha}$ starting from a henceforth fixed set of simple roots $I$.


\medskip

\subsection{Quantum groups}

A well-known $q$-deformation of (the universal enveloping algebra of) a finite type Lie algebra $\fg$ is the Drinfeld--Jimbo quantum group $\uu$. This is an associative $\BQ(q)$-algebra, which also admits a triangular decomposition
$$
  \uu = \uup \otimes \uuo \otimes \uum
$$
Lusztig (\cite{Lu}) constructed \emph{root vectors}
\begin{equation}
\label{eqn:root vectors}
  \{f_{\alpha}\}_{\alpha \in \Delta^+} \in \uum
\end{equation}
which are $q$-deformations of $\sX_{-\alpha} \in \fn^-$. The input (respectively tools) for Lusztig's construction is that of a reduced decomposition of the longest word in the Weyl group associated to $\fg$ (respectively the action of the corresponding braid group on $\uu$). Such a choice is equivalent (\cite{P}) to a total order of the set of positive roots
\begin{equation}
\label{eqn:total order intro}
  \Delta^+ = \{\cdots < \alpha < \cdots\}
\end{equation}
which is \emph{convex}, in the sense that for all pairs of positive roots $\alpha < \beta$ whose sum $\alpha+\beta$ is also a root, we have
\begin{equation}
\label{eqn:convex intro}
  \alpha < \alpha+\beta < \beta
\end{equation}
The following commutation relations were proved by Levendorskii--Soibelman (\cite{LS}) for all positive roots $\alpha < \beta$ (the details of the proof can be found in~\cite[\S9.3]{DCP})
\begin{equation}
\label{eqn:commutation intro}
  f_{\alpha} f_{\beta} - q^{-(\alpha,\beta)} f_{\beta} f_{\alpha} \, \in
  \bigoplus_{\beta > \gamma_1 \geq \dots \geq \gamma_t > \alpha}^{\gamma_1+\dots+\gamma_t=\alpha+\beta} \BZ[q,q^{-1}] \cdot f_{\gamma_1} \dots f_{\gamma_t}
\end{equation}
In particular, if $\alpha < \beta$ form a \emph{minimal pair} (i.e.\ $\alpha+\beta$ is a root and there do not exist roots $\alpha',\beta'$ such that $\alpha < \alpha' < \beta' < \beta$ and $\alpha'+\beta' = \alpha+\beta$) then we have
\begin{equation}
\label{eqn:commutation intro minimal}
  f_{\alpha} f_{\beta} - q^{-(\alpha,\beta)} f_{\beta} f_{\alpha} \in \BZ[q,q^{-1}]^\times \cdot f_{\alpha+\beta}
\end{equation}
This formula can be used to successively construct (up to constant multiples) the root vectors $f_{\alpha}$ starting from a set of simple roots $I$.


\medskip

\subsection{Quantum loop groups}

The main object of our study is the quantum loop group $\UU$, which is an associative algebra with a triangular decomposition
$$
  \UU = \UUp \otimes \UUo \otimes \UUm
$$
Whereas $\uum$ was generated by the symbols $\{f_i\}_{i \in I}$, the negative half $\UUm$ of the quantum loop group is generated by the coefficients of formal series
$$
  f_i(x) = \sum_{d \in \BZ} \frac {f_{i,d}}{x^d},  \qquad \forall\, i\in I
$$
(the explicit relations between these generators are well-known, and we recall them in Definition \ref{def:quantum loop} below). By generalizing Lusztig's construction, Ding-Khoroshkin (\cite{DK}) defined so-called \emph{fused currents} for every positive root
$$
  \tf_{\alpha}(x) = \sum_{d \in \BZ} \frac {\tf_{\alpha,d}}{x^d}, \qquad \forall\, \alpha\in \Delta^+
$$
whose coefficients are infinite sums which converge in certain representations of quantum loop groups. In Definition \ref{def:completion} below, we will define a completion of the negative half $\UUm$ in which we conjecture that all the coefficients $\tf_{\alpha,d}$ live, see Conjecture~\ref{conj:dk}, thus providing a formal algebraic treatment of the main construction of \cite{DK}.

\medskip
\noindent
Then it remains to ask whether the analogues of the properties \eqref{eqn:commutation intro} and \eqref{eqn:commutation intro minimal} hold for fused currents. For instance, if $\alpha < \beta$ is a minimal pair of positive roots, we posit in Conjecture \ref{conj:comm} that
\begin{equation}
\label{eqn:comm intro}
  \tf_{\alpha}(x) \tf_{\beta}(y) - \tf_{\beta}(y) \tf_{\alpha}(x) \cdot \frac {xq^{\tt} - yq^{(\alpha,\beta)}}{xq^{\tt+(\alpha,\beta)}-y} =
  c(q) \cdot \delta\left( \frac {xq^{\tt+(\alpha,\beta)}}y \right)  \tf_{\alpha+\beta}(xq^u)
\end{equation}
for some $u,\tt \in \BZ$, $c(q) \in \BZ[q,q^{-1}]^\times$ that depend on $\alpha$ and $\beta$ (let $\delta(x) = \sum_{d \in \BZ} x^d$)~\footnote{As $q \to 1$, the quantum loop group $\UU$ converges to the universal enveloping algebra $U(L\mathfrak{g}) = U(\fg[z,z^{-1}])$, with the Lie bracket on the loop algebra $\fg[z,z^{-1}]$ defined via $[az^k, bz^l] = [a,b]z^{k+l}$ for $a,b\in \fg$, $k,l\in \BZ$. One expects the limit of $\tf_{\alpha,d}$ to be $\sX_{-\alpha}z^d$, so that~\eqref{eqn:comm intro} converges~to
$$
  \left[\sX_{-\alpha} \delta \left(\frac zx \right), \sX_{-\beta} \delta \left(\frac zy \right) \right] =
  c(1)\cdot \delta \left(\frac xy \right) \sX_{-\alpha-\beta} \delta \left(\frac zx \right)
$$
Extracting the coefficient of $x^{-k}y^{-l}$ from the formula above gives rise to \eqref{eqn:chevalley}.}.


\medskip

\subsection{Specialization maps}

To understand fused currents explicitly, we will use the dual shuffle algebra picture (studied in the quantum loop group setting by Enriquez in \cite{E1}, motivated by a construction of Feigin--Odesskii in \cite{FO}). We will recall the shuffle algebra in detail in Section \ref{sec:shuffle}, but in a nutshell, it is defined by
$$
  \CS = \bigoplus_{\bk = \sum_{i \in I} k_i \bs^i \in \nn} \CS_{\bk}
$$
$$
  \CS_{\bk} = \Big\{\text{certain rational functions } R(z_{i1},\dots,z_{ik_i})_{i \in I} \Big\}
$$
In the present paper, $\BN = \{0,1,2,\dots\}$ and $\bs^i \in \nn$ is the $I$-tuple with a single 1 at position $i \in I$ and 0 everywhere else. The vector space $\CS$ is made into an algebra via the shuffle product \eqref{eqn:mult}. The shuffle algebra $\CS$ thus defined is relevant to us because it is isomorphic to the positive half of the quantum loop group
$$
  \CS \simeq \UUp
$$
as well as dual to the negative half
$$
  \CS \otimes \UUm \xrightarrow{\langle \cdot, \cdot \rangle} \BQ(q)
$$
For every $\alpha \in \Delta^+$ we may describe the fused currents $\tf_{\alpha}(x)$ by the dual linear functional, which we will call a \emph{specialization map}
$$
  \CS_{\alpha} \xrightarrow{\tspec^{(x)}_{\alpha}} \BQ(q)[x,x^{-1}], \qquad \tspec^{(x)}_\alpha(R) = \Big \langle R, \tf_{\alpha}(x) \Big \rangle
$$
In Subsection \ref{sub:specialization}, we conjecture that $\tspec_{\alpha}^{(x)}(R)$ is equal (up to a scalar prefactor and a power of $x$) to a certain derivative of the rational function $R$ when its variables are specialized according to (write $\alpha = \sum_{i\in I} k_i\bs^i$ for suitable $k_i \in \BN$)
$$
  \big\{z_{ib} = xq^{\sigma_{ib}}\big\}^{i \in I}_{1\leq b \leq k_i}
$$
for certain $\sigma_{ib} \in \BZ$. The specialization maps strongly depend on the total order \eqref{eqn:total order intro}, and we do not yet have a complete understanding of $\tspec^{(x)}_{\alpha}$ in full generality.


\medskip

\subsection{Quivers}

Beside laying out the general expectations of fused currents and specialization maps (as in the preceding paragraph), the main purpose of the present paper is to construct specialization maps in the particular case when the total order \eqref{eqn:total order intro} is a refinement of the Auslander--Reiten partial order on $\Delta^+$ developed in \cite{Ri} (see Section \ref{sec:quivers} for an overview). Explicitly, the following is our main result.

\medskip

\begin{theorem}\label{thm:main intro}
Let $Q$ be any orientation of a type ADE Dynkin diagram, and let us fix any total order \eqref{eqn:total order intro} that refines the Auslander--Reiten partial order of the positive roots induced by $Q$. For every positive root $\alpha = \sum_{i \in I}k_i \bs^i\in \Delta^+$, define
\begin{equation}
\label{eqn:spec intro}
  \CS_{\alpha} \xrightarrow{\emph{spec}^{(x)}_{\alpha}} \BQ(q)[x,x^{-1}]
\end{equation}
$$
  \espec^{(x)}_\alpha(R) = \gamma^{(x)}_\alpha \cdot R(z_{i1},\dots,z_{ik_i}) \Big|_{z_{ib} \mapsto x q^{\tau(i)}, \forall i,b}
$$
(see Definition \ref{def:specialization ar} for the specific factor $\gamma^{(x)}_\alpha \in \BQ(q^{1/2})^\times \cdot x^{\BZ}$ that appears in the formula above) where $\tau\colon I\to \BZ$ is a height function, i.e. satisfies $\tau(i) = \tau(j)+1$ if there exists an edge in $Q$ from $i$ to $j$. If we let $f_{\alpha}(x)$ be the dual of the specialization maps
$$
  \espec^{(x)}_\alpha(R) = \Big \langle R, f_{\alpha}(x) \Big \rangle
$$
then for any minimal pair $\alpha < \beta$ of positive roots, we have
\begin{equation}
\label{eqn:fused intro equality}
  f_{\alpha}(x) f_{\beta}(y) -  f_{\beta}(y)f_{\alpha}(x)  \frac {xq-yq^{-1}}{x-y} = \delta\left(\frac xy\right) f_{\alpha+\beta}(x)
\end{equation}
\end{theorem}

\medskip
\noindent
We conjecture that $f_{\alpha}(x)$ (respectively $\spec_\alpha^{(x)}$) equals $\tf_{\alpha}(x)$ (respectively $\tspec_\alpha^{(x)}$) up to a constant multiple, which would establish formula \eqref{eqn:comm intro} in the particular case of ADE types where the convex order on the set of positive roots is a refinement of the Auslander--Reiten partial order.


\medskip

\subsection{Perspectives}
While the initial motivation for our work was the exploration of the fused currents $\tf_{\alpha}(x)$ of \cite{DK}, we propose that specialization maps can provide an altogether alternative definition of these objects. The latter viewpoint may be applied in principle to quantum loop groups of any Kac--Moody type, where the original definition of fused currents is unavailable due to the non-existence of the braid group action. Thus, the approach we propose is to define specialization maps
\begin{equation}
\label{eqn:specialization perspectives}
  \CS_{\alpha} \xrightarrow{\text{spec}^{(x)}_{\alpha}} \BQ(q)[x,x^{-1}]
\end{equation}
directly, and then redefine the fused currents $f_{\alpha}(x)$ as their duals
\begin{equation}
\label{eqn:duality perspectives}
  \spec^{(x)}_\alpha(R) = \Big \langle R, f_{\alpha}(x) \Big \rangle
\end{equation}
The specializations \eqref{eqn:specialization perspectives} should be defined so that one has a commutation relation
\begin{multline*}
  f_{\alpha}(x) f_{\beta}(y) - f_{\beta}(y) f_{\alpha}(x) \Big(\text{a rational function in }x,y \Big) = \\
  \Big(\text{a sum of } \delta \text{ functions and products of }f_\gamma \Big)
\end{multline*}
As the contents of the present paper show, the above construction features interesting combinatorics even when $\fg$ is of finite type, and it is bound to be challenging (yet doable in our opinion) for $\fg$ of affine type. For $\fg$ of general type, we expect the combinatorics of the specialization maps to be extremely difficult.


\medskip

\subsection{Acknowledgements}

Both authors would like to thank Mathematisches Forschungsinstitut Oberwolfach (Oberwolfach, Germany) and Centre International de Rencontres Mathématiques (Luminy, France) for their hospitality and wonderful working conditions in the Summer 2023 while the present work was being carried out. A.N.\ gratefully acknowledges support from the NSF grant DMS-$1845034$, the MIT Research Support Committee, and the PNRR grant CF 44/14.11.2022 titled “Cohomological Hall algebras of smooth surfaces and applications”. A.T.\ gratefully acknowledges NSF grants DMS-$2037602$ and DMS-$2302661$.


\bigskip

\section{Quantum loop groups and shuffle algebras}
\label{sec:shuffle}


\medskip

\subsection{Definitions}

Fix a Lie algebra $\fg$ of finite type and a decomposition $\Delta = \Delta^+ \sqcup \Delta^-$ of the corresponding root system into positive and negative roots. We also fix a set of simple roots $\{\alpha_i\}_{i \in I} \subset \Delta^+$ and consider the inner product $(\cdot, \cdot)$ on the root lattice. The Cartan matrix corresponding to $\fg$ is
$$
  (a_{ij})_{i,j \in I} \quad \text{with} \quad a_{ij} = \frac {2(\alpha_i,\alpha_j)}{(\alpha_i,\alpha_i)}
$$
In what follows, let $q$ be a formal variable, set $q_i = q^{\frac {(\alpha_i,\alpha_i)}2}$ for all $i \in I$, and consider the generating series
$$
  e_i(x) = \sum_{d \in \BZ} \frac {e_{i,d}}{x^d}, \qquad
  f_i(x) = \sum_{d \in \BZ} \frac {f_{i,d}}{x^d}, \qquad
  \ph^\pm_i(x) = \sum_{d = 0}^\infty \frac {\ph^\pm_{i,d}}{x^{\pm d}}
$$
and the formal delta function $\delta(x) = \sum_{d \in \BZ} x^d$. For any $i,j \in I$, set
\begin{equation}
\label{eqn:zeta}
  \zeta_{ij} \left(\frac xy \right) = \frac {x - yq^{-(\alpha_i,\alpha_j)}}{x - y}
\end{equation}
We now recall the definition of the quantum loop group (new Drinfeld realization).

\medskip

\begin{definition}\label{def:quantum loop}
The quantum loop group associated to $\fg$ is:
$$
  \UU = \BQ(q) \Big \langle e_{i,d}, f_{i,d}, \ph_{i,d'}^\pm \Big \rangle_{i \in I}^{d \in \BZ, d' \geq 0} \Big/
  \text{relations \eqref{eqn:rel 0 affine}--\eqref{eqn:rel 3 affine}}
$$
where we impose the following relations for all $i,j \in I$:
\begin{equation}
\label{eqn:rel 0 affine}
  e_i(x) e_j(y) \zeta_{ji} \left(\frac yx \right) =\, e_j(y) e_i(x) \zeta_{ij} \left(\frac xy \right)
\end{equation}
\begin{multline}
\label{eqn:rel 1 affine}
  \sum_{\sigma \in S(1-a_{ij})} \sum_{k=0}^{1-a_{ij}} (-1)^k {1-a_{ij} \choose k}_{q_i} \cdot  \\
  \qquad \qquad \qquad e_i(x_{\sigma(1)})\dots e_i(x_{\sigma(k)}) e_j(y) e_i(x_{\sigma(k+1)}) \dots e_i(x_{\sigma(1-a_{ij})}) = 0,
  \quad \mathrm{if}\ i\ne j
\end{multline}
\begin{equation}
\label{eqn:rel 2 affine}
  \ph_j^\pm(y) e_i(x) \zeta_{ij} \left(\frac xy \right) = e_i(x) \ph_j^\pm(y) \zeta_{ji} \left(\frac yx \right)
\end{equation}
\begin{equation}
\label{eqn:rel 2 affine bis}
  \ph_{i}^{\pm}(x) \ph_{j}^{\pm'}(y) = \ph_{j}^{\pm'}(y) \ph_{i}^{\pm}(x), \quad
  \ph_{i,0}^+ \ph_{i,0}^- = 1
\end{equation}
as well as the opposite relations with $e$'s replaced by $f$'s, and finally the relation:
\begin{equation}
\label{eqn:rel 3 affine}
  \left[ e_i(x), f_j(y) \right] =
  \frac {\delta_i^j \delta \left(\frac xy \right)}{q_i - q_i^{-1}} \cdot  \Big( \ph_i^+(x) - \ph_i^-(y) \Big)
\end{equation}
\end{definition}

\medskip
\noindent
The algebra $\UU$ is $\zz \times \BZ$-graded via
$$
  \deg e_{i,d} = (\bs^i,d), \qquad \deg \ph_{i,d}^\pm = (0,\pm d), \qquad \deg f_{i,d} = (-\bs^i,d)
$$
where $\bs^i = \underbrace{(0,\dots,0,1,0,\dots,0)}_{1\text{ on the }i\text{-th position}}$. We have the triangular decomposition (\cite{H})
\begin{equation*}
  \UU = \UUp \otimes \UUo \otimes \UUm
\end{equation*}
into the subalgebras generated by $e_{i,d}$, $\ph_{i,d}^{\pm}$, $f_{i,d}$, respectively. Note the isomorphism
\begin{equation}
\label{eqn:iso halves}
  \UUp \iso \UUm
\end{equation}
determined by sending $e_{i,d} \mapsto f_{i,-d}$ for all $i \in I$, $d \in \BZ$ (here, we use that $\UUp$ is generated by $e_{i,d}$ with the defining relations~(\ref{eqn:rel 0 affine},~\ref{eqn:rel 1 affine}), and similarly for~$\UUm$).


\medskip

\subsection{The shuffle algebra}
\label{sub:FO-algebra}

We now recall the trigonometric degeneration (\cite{E1, E2}) of the Feigin--Odesskii shuffle algebra (\cite{FO}) of type $\fg$. Consider the vector space of rational functions
\begin{equation*}
  \CV \ = \bigoplus_{\bk = \sum_{i \in I} k_i \bs^i \in \nn} \BQ(q)(\dots,z_{i1},\dots,z_{ik_i},\dots)^{\text{sym}}_{i \in I}
\end{equation*}
which are \emph{color-symmetric}, meaning that they are symmetric in the variables $z_{i1},\dots,z_{ik_i}$ for each $i \in I$ separately (the terminology is inspired by the fact that $i \in I$ is called the color of the variable $z_{ib}$, for any $i \in I$ and $b \geq 1$). We make the vector space $\CV$ into a $\BQ(q)$-algebra via the following shuffle product:
\begin{equation}
\label{eqn:mult}
  F(\dots, z_{i1}, \dots, z_{i k_i}, \dots) * G(\dots, z_{i1}, \dots,z_{i l_i}, \dots) = \frac 1{\bk! \cdot \bl!}\,\cdot
\end{equation}
$$
  \textrm{Sym} \left[ F(\dots, z_{i1}, \dots, z_{ik_i}, \dots) G(\dots, z_{i,k_i+1}, \dots, z_{i,k_i+l_i}, \dots)
               \prod_{i,j \in I} \prod_{b \leq k_i}^{c > k_j} \zeta_{ij} \left( \frac {z_{ib}}{z_{jc}} \right) \right]
$$
In \eqref{eqn:mult}, $\sym$ denotes symmetrization with respect to the:
\begin{equation*}
  (\bk+\bl)! := \prod_{i\in I} (k_i+l_i)!
\end{equation*}
permutations of the variable sets $\{z_{i1}, \dots, z_{i,k_i+l_i}\}$ for each $i$ independently.

\medskip

\begin{definition}
\label{def:shuf}
(\cite{E1,E2}, inspired by \cite{FO})
The shuffle algebra $\CS$ is the subspace of $\CV$ consisting of rational functions of the form:
\begin{equation}
\label{eqn:shuf}
  R(\dots,z_{i1},\dots,z_{ik_i},\dots) =
  \frac {r(\dots,z_{i1},\dots,z_{ik_i},\dots)}
  {\prod^{\text{unordered}}_{\{i \neq i'\} \subset I} \prod_{1\leq b \leq k_i}^{1\leq b' \leq k_{i'}} (z_{ib} - z_{i'b'})}
\end{equation}
where $r$ is a color-symmetric Laurent polynomial that satisfies the wheel conditions:
\begin{equation}
\label{eqn:wheel}
  r(\dots, z_{ib}, \dots)\Big|_{(z_{i1},z_{i2},z_{i3}, \dots, z_{i,1-a_{ij}}) \mapsto (w, w q_i^{2}, wq_i^{4}, \dots, w q_i^{-2a_{ij}}),\, z_{j1} \mapsto w q_i^{-a_{ij}}} =\, 0
\end{equation}
for any distinct $i, j \in I$.
\end{definition}

\medskip
\noindent
It is elementary to show that $\CS$ is closed under the shuffle product~\eqref{eqn:mult}, and is thus an algebra. Because of~(\ref{eqn:wheel}), any $r$ as in \eqref{eqn:shuf} is actually divisible by:
$$
  \prod^{\text{unordered}}_{\{i \neq i'\} \subset I:a_{ii'}=0} \prod_{1\leq b \leq k_i}^{1\leq b' \leq k_{i'}} (z_{ib} - z_{i'b'})
$$
Therefore, rational functions $R$ satisfying~(\ref{eqn:shuf},~\ref{eqn:wheel}) can only have simple poles on the diagonals $z_{ib}=z_{i'b'}$ with adjacent $i,i'\in I$, that is, such that $a_{ii'}<0$.


\medskip

\subsection{Shuffle algebras vs quantum loop groups}
\label{sub:qaff vs FO}

The algebra $\CS$ is graded by $\bk = \sum_{i \in I} k_i \bs^i \in \nn$ that encodes the number of variables of each color, and by the total homogeneous degree $d\in \BZ$. We write:
$$
  \deg R = (\bk,d)
$$
and say that $\CS$ is $(\nn \times \BZ)$-graded. We will denote the graded pieces by:
$$
  \CS = \bigoplus_{\bk \in \nn} \CS_{\bk} \ \ \qquad \text{and} \qquad \ \
  \CS_{\bk} = \bigoplus_{d \in \BZ} \CS_{\bk, d}
$$
We now give the first connection between shuffle algebras and quantum loop groups.

\medskip

\begin{theorem}
\label{thm:iso}
(\cite{E1} for the homomorphism, \cite{NT} for the isomorphism)
The assignment $e_{i,d}\mapsto z_{i1}^d \in \CS_{\bs^i,d}$ gives rise to an algebra isomorphism
\begin{equation}
\label{eqn:iso}
  \Upsilon\colon \UUp \iso \CS
\end{equation}
\end{theorem}

\medskip
\noindent
In particular, Theorem \ref{thm:iso} implies that $\CS$ is generated by the monomials $\{z_{i1}^d\}_{i \in I}^{d \in \BZ}$ under the shuffle product~\eqref{eqn:mult}. Besides being isomorphic to the positive half of quantum loop groups (as above), shuffle algebras are also dual to the negative half of quantum loop groups, as we recall next. In what follows, let $Dz = \frac {dz}{2\pi i z}$.

\medskip

\begin{theorem}
\label{thm:dual}
(\cite{E1} for the construction of the pairing, \cite{NT} for its non-degeneracy) There is a non-degenerate pairing
\begin{equation}
\label{eqn:pairing}
  \CS \otimes \UUm \xrightarrow{\langle \cdot, \cdot \rangle} \BQ(q)
\end{equation}
given by
\begin{equation}
\label{eqn:pair shuf generators}
  \Big \langle R, f_{i_1,-d_1} \dots f_{i_k,-d_k} \Big \rangle = \int_{|z_{1}| \ll \dots \ll |z_{k}|}
      \frac {R(z_{1}, \dots,z_{k}) z_{1}^{-d_1} \dots z_{k}^{-d_k}}
          {\prod_{1 \leq a < b \leq k} \zeta_{i_a i_b} (z_{a}/z_{b})} \prod_{a=1}^k D z_{a}
\end{equation}
for any $k\in \BN,\, i_1, \dots,i_k \in I,\,  d_1, \dots, d_k \in \BZ,\, R \in \CS_{\bs_{i_1}+\dots+\bs_{i_k},d_1+\dots+d_k}$ (all pairings between elements of non-opposite degrees are set to be $0$). In the right-hand side of~\eqref{eqn:pair shuf generators}, we plug each variable $z_{a}$ into an argument of color $i_a$ of the function $R$; since the latter is symmetric, the result is independent of any choices made. The symbol $\int$ in \eqref{eqn:pair shuf generators} refers to the contour integral over concentric circles centered at the origin of the complex plane (the notation $|z_a| \ll |z_b|$ means that the these circles are very far away from each other when compared to the poles of the integrand).
\end{theorem}

\medskip
\noindent
Note that the pairing \eqref{eqn:pair shuf generators} differs from that of \cite[(5.16, 5.17)]{NT} by a scalar multiple, but we make the present choice to keep our formulas as clear as possible.


\medskip

\subsection{The slope filtration}

We will now construct a filtration of $\CS$ by a notion of slope $\mu \in \BQ$, with the ultimate goal of defining a completion of the shuffle algebra. By Theorem \ref{thm:iso}, this will lead to a specific completion of the positive half of the quantum loop group, hence also of the negative half by the isomorphism \eqref{eqn:iso halves}. The coefficients of fused currents are expected to lie in this completion.

\medskip
\noindent
Given $\bk = \sum_{i \in I} k_i\bs^i$ and $\bl = \sum_{i \in I} l_i \bs^i$, we will write $\bl \leq \bk$ if $l_i \leq k_i$ for all $i \in I$. We also write $|\bl| = \sum_{i \in I} l_i$. The following is a close relative of \cite[Definition 3.3]{N1}.

\medskip

\begin{definition}\label{def:slope}
For any $\mu \in \BQ$, we say that $R \in \CS_\bk$ has slope $\leq \mu$ if
\begin{equation}
\label{eqn:slope}
  \lim_{\xi \rightarrow \infty} \frac {R(\dots,\xi z_{i1},\dots,\xi z_{il_i},z_{i,l_{i+1}},\dots,z_{ik_i},\dots)}{\xi^{\mu |\bl|}}
\end{equation}
is finite for all $\bl \leq \bk$. Let $\CS_{\bk,\leq \mu} \subset \CS_\bk$ denote the set of such elements and set
$$
  \CS_{\leq \mu}=\bigoplus_{\bk\in \nn} \CS_{\bk,\leq \mu}
$$
\end{definition}

\noindent
One can show that $\CS_{\leq \mu}$ is a subalgebra of $\CS$ (cf.~\cite[Proposition 2.3]{N2} for the argument in a completely analogous setup). Moreover, the set
$$
  \CB^+_{\mu} = \mathop{\bigoplus_{(\bk,d) \in \nn \times \BZ}}_{d = \mu|\bk|} \CB_{\bk,d} \quad \text{with} \quad  \CB_{\bk,d} = \CS_{\leq \frac d{|\bk|}} \cap \CS_{\bk,d}
$$
is also a subalgebra of $\CS$, which we will refer to as a \emph{slope subalgebra}. Note that each $\CB_{\bk,d}$ is a finite-dimensional $\BQ(q)$-vector space (cf.\ the proof of Lemma~\ref{lem:path}). By analogy with \cite[Theorem 1.1]{N1}, one can show that the multiplication map
\begin{equation}
\label{eqn:factorization 1}
  m\colon \bigotimes^{\rightarrow}_{\mu \in \BQ} \CB^+_{\mu} \iso \CS
\end{equation}
is an isomorphism of vector spaces (the arrow $\rightarrow$ refers to taking the product in increasing order of $\mu$). More generally, the subalgebras $\CS_{\leq \nu}$ of $\CS$ admit factorizations as in \eqref{eqn:factorization 1}, but with $\mu$ only running over the set $(-\infty, \nu]$.

\medskip

\begin{proposition}
\label{prop:commute}
For any $a_1 \in \CB^+_{\nu_1}, \dots, a_t \in \CB_{\nu_t}^+$, we have
\begin{equation}
\label{eqn:commute}
  a_1 \dots a_t \ \in \bigotimes_{\mu \in [\min(\nu_s), \max(\nu_s)]}^{\rightarrow} \CB^+_{\mu}
\end{equation}
\end{proposition}

\medskip

\begin{proof}
Since $\CS_{\leq \mu}$ is an algebra for any $\mu$, the fact that $a_1,\dots,a_t$ lie in $\CS_{\leq \max(\nu_s)}$ implies that $a_1\dots a_t \in \CS_{\leq \max(\nu_s)}$. By the sentence following \eqref{eqn:factorization 1}, we have
\begin{equation}
\label{eqn:commute 1}
  a_1 \dots a_t \ \in \bigotimes_{\mu \in (-\infty,\max(\nu_s)]}^{\rightarrow} \CB^+_{\mu}
\end{equation}
However, completely analogously to Definition \ref{def:slope}, one can define the subalgebra $\CS_{\geq \mu} \subset \CS$ with $\CS_{\geq \mu}\cap \CS_{\bk}$ consisting of rational functions such that
\begin{equation}
\label{eqn:slope two}
  \lim_{\xi \rightarrow 0} \frac {R(\dots,\xi z_{i1},\dots,\xi z_{il_i},z_{i,l_{i+1}},\dots,z_{ik_i},\dots)}{\xi^{\mu |\bl|}}
\end{equation}
is finite for all $\bl \leq \bk$. Moreover, for any $(\bk,d) \in \nn \times \BZ$, it is clear that
$$
  \CB_{\bk,d} = \CS_{\geq \frac d{|\bk|}} \cap \CS_{\bk,d}
$$
because properties \eqref{eqn:slope} and \eqref{eqn:slope two} are equivalent for a rational function of homogeneous degree $d=\mu|\bk|$ in $\bk$ variables. Hence, by analogy with \eqref{eqn:commute 1}, we~obtain
\begin{equation}
\label{eqn:commute 2}
  a_1\dots a_t \ \in \bigotimes_{\mu \in [\min(\nu_s),\infty)}^{\rightarrow} \CB^+_{\mu}
\end{equation}
Thus, combining~(\ref{eqn:commute 1},~\ref{eqn:commute 2}) with the fact that the map \eqref{eqn:factorization 1} is an isomorphism, we conclude \eqref{eqn:commute}.
\end{proof}


\medskip

\subsection{Slopes and the pairing}

By abuse of notation, we will also refer to $\CB^+_{\mu}$ as a subalgebra of $\UUp$ via the isomorphism \eqref{eqn:iso}. Therefore, we obtain the following analogue of \eqref{eqn:factorization 1}
\begin{equation*}
  \bigotimes^{\rightarrow}_{\mu \in \BQ} \CB^+_{\mu} \iso \UUp
\end{equation*}
Using the isomorphism \eqref{eqn:iso halves}, we also have a factorization
\begin{equation}
\label{eqn:factorization 3}
  \bigotimes^{\rightarrow}_{\mu \in \BQ} \CB^-_{\mu} \iso \UUm
\end{equation}
where $\CB^-_\mu$ refers to the image of $\CB^+_\mu$ under \eqref{eqn:iso halves} (we will denote the graded pieces of $\CB^-_\mu$ by $\CB_{-\bk,-d}$, with $\mu=\frac{d}{|\bk|}$, in order to differentiate them from those of $\CB^+_\mu$). The following result explains that the pairing \eqref{eqn:pairing} is determined by its values on the slope subalgebras (cf.~\cite[Proposition 3.12]{N1} for the argument in an analogous~setup).

\medskip

\begin{proposition}
\label{prop:pair slopes}
For any $\{b^+_\mu \in \CB^+_\mu, b^-_\mu \in \CB^-_\mu\}_{\mu \in \BQ}$ (with almost all of the $b^+_\mu, b^-_\mu$ being 1) we have the following formula for the pairing \eqref{eqn:pairing}
$$
  \left \langle \prod_{\mu \in \BQ}^{\rightarrow} b^+_\mu, \prod_{\mu \in \BQ}^{\rightarrow} b^-_\mu \right \rangle = \prod_{\mu \in \BQ} \langle b^+_\mu, b^-_\mu \rangle
$$
\end{proposition}


\medskip

\subsection{Products and paths}

Let us consider a product
\begin{equation}
\label{eqn:product p}
  \pi = \prod_{\mu \in \BQ}^\rightarrow b_{\mu} \in \UUm_{-\bk,-d}
\end{equation}
with $b_\mu \in \CB^-_\mu$ for all $\mu \in \BQ$, such that almost all of the $b_{\mu}$ are 1. We further assume that each $b_\mu$ is homogeneous of some degree $(-\bk_{\mu}, -d_{\mu}) \in (-\nn) \times \BZ$, almost all of which will be $(0,0)$. To such a product~\eqref{eqn:product p} we associate the sequence of vectors
$$
  \Big( \dots, ( |\bk_{\mu}|, d_{\mu} ) , \dots \Big)_{\mu \in \BQ} \subset \BN \times \BZ
$$
Since almost all of the vectors in the sequence above are $(0,0)$, by placing these vectors head-to-tail (in increasing order of $\mu \in \BQ$), we obtain a convex path $p = \text{path}(\pi)$ in the lattice plane from the origin $(0,0)$ to $(|\bk|,d)$, see an example below. We call $(|\bk|,d)$ the \emph{size} of the path $p$ and the vectors $(|\bk_{\mu}|, d_{\mu} )$ the \emph{legs} of the path.

\medskip

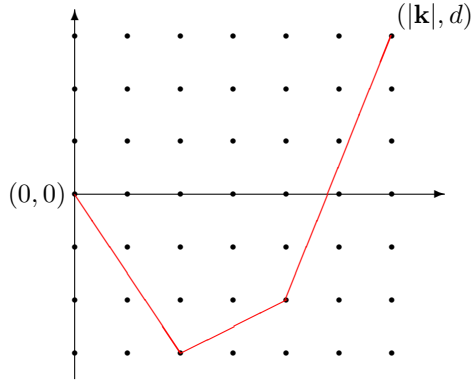
\begin{picture}(100,150)(-110,-17)
\label{pic:par}

\put(0,0){\circle*{2}}\put(20,0){\circle*{2}}\put(40,0){\circle*{2}}\put(60,0){\circle*{2}}\put(80,0){\circle*{2}}\put(100,0){\circle*{2}}\put(120,0){\circle*{2}}\put(0,20){\circle*{2}}\put(20,20){\circle*{2}}\put(40,20){\circle*{2}}\put(60,20){\circle*{2}}\put(80,20){\circle*{2}}\put(100,20){\circle*{2}}\put(120,20){\circle*{2}}\put(0,40){\circle*{2}}\put(20,40){\circle*{2}}\put(40,40){\circle*{2}}\put(60,40){\circle*{2}}\put(80,40){\circle*{2}}\put(100,40){\circle*{2}}\put(120,40){\circle*{2}}\put(0,60){\circle*{2}}\put(20,60){\circle*{2}}\put(40,60){\circle*{2}}\put(60,60){\circle*{2}}\put(80,60){\circle*{2}}\put(100,60){\circle*{2}}\put(120,60){\circle*{2}}\put(0,80){\circle*{2}}\put(20,80){\circle*{2}}\put(40,80){\circle*{2}}\put(60,80){\circle*{2}}\put(80,80){\circle*{2}}\put(100,80){\circle*{2}}\put(120,80){\circle*{2}}\put(0,100){\circle*{2}}\put(20,100){\circle*{2}}\put(40,100){\circle*{2}}\put(60,100){\circle*{2}}\put(80,100){\circle*{2}}\put(100,100){\circle*{2}}\put(120,100){\circle*{2}}\put(0,120){\circle*{2}}\put(20,120){\circle*{2}}\put(40,120){\circle*{2}}\put(60,120){\circle*{2}}\put(80,120){\circle*{2}}\put(100,120){\circle*{2}}\put(120,120){\circle*{2}}

\put(0,-10){\vector(0,1){140}} 
\put(0,60){\vector(1,0){140}} 

\textcolor{red}{\put(0,60){\line(2,-3){40}}} 
\textcolor{red}{\put(37,0){\line(2,1){40}}} 
\textcolor{red}{\put(73,20){\line(2,5){40}}} 

\put(-32,57){$(0,0)$} 
\put(115,125){$(|\bk|,d)$} 

\end{picture}

\medskip

\begin{lemma}
\label{lem:path}
For any convex path $p$, the vector space $\UUm_p$ spanned by products \eqref{eqn:product p} with $\emph{path}(\pi) = p$ is finite-dimensional.
\end{lemma}

\medskip

\begin{proof}
This is an immediate consequence of the fact that the graded pieces of the algebras $\CB^-_\mu$ are finite-dimensional, itself a consequence of the fact that the space of Laurent polynomials in $|\bk|$ variables of fixed total homogeneous degree, but with degree bounded from above in each variable, is finite-dimensional.
\end{proof}

\noindent
Consider any collection of homogeneous elements
$$
  a_1, \dots, a_t \in \UUm
$$
By \eqref{eqn:factorization 3}, each of these elements can be written as
$$
  a_s = \pi_s^1+ \pi_s^2 + \dots
$$
where each $\pi_s^1,\pi_s^2,\dots$ is of the form \eqref{eqn:product p}. Then
$$
  P_s = \Big\{ \text{path}(\pi_s^1), \text{path}(\pi_s^2),\dots \Big\}
$$
is a finite set of convex paths of size $(|\bk_s|,d_s)$, and we will write $P$ for the finite collection of concatenations of any path from $P_1$ with any path from $P_2$ \dots with any path from $P_t$.

\medskip

\begin{proposition}
\label{prop:olivier}
With $a_1,\dots,a_t$ as above, the product $a_1\dots a_t$ is a finite sum of products $\pi$ of the form \eqref{eqn:product p}, such that $\emph{path}(\pi)$ lies below some path in $P$.
\end{proposition}

\medskip
\noindent
The notion ``lie below" in $\BZ^2$ naturally refers to having smaller than or equal second coordinate: we say that a path $p$ lies below a path $p'$ if for any $x \in \BR$, those $y$ such that $(x,y) \in p$ are smaller than or equal to those $y'$ such that $(x,y') \in p'$. In what follows, we only sketch the proof of Proposition \ref{prop:olivier}, and refer the reader to~\cite[Section 5]{S} for the full and original details of this kind of argument.

\begin{proof}
It suffices to assume that $a_s \in \CB_{\mu_s}^-$ for various $\mu_s \in \BQ$, so each $P_s$ consists of a single leg and thus $P$ consists of a single path $p$. If we have $\mu_s > \mu_{s+1}$ for some $s$ (which corresponds to two consecutive legs in $p$ which violate convexity), then we may use (the image under the isomorphism \eqref{eqn:iso halves} of) Proposition \ref{prop:commute} to write
\begin{equation}
\label{eqn:temp}
  a_s a_{s+1} = \sum (\pi \text{ as in \eqref{eqn:product p}})
\end{equation}
The $\pi$'s that appear in the right-hand side of \eqref{eqn:temp} correspond to convex paths with the same endpoints as the \underline{concave} $\text{path}(a_sa_{s+1})$. In particular, the aforementioned convex paths lie strictly below the aforementioned concave path and moreover the area between these convex and concave paths is a positive integer multiple of $\frac 12$. Therefore, if we plug the right-hand side of \eqref{eqn:temp} in the middle of the product
$$
  a_1 \dots a_s a_{s+1} \dots a_t
$$
then we are replacing a product corresponding to the path $p$ by a sum of products corresponding to paths lying below $p$. However, due to Proposition~\ref{prop:commute}, all of these paths still lie above the \emph{convexification} $p^\sharp$ of the path $p$,  which is the convex path obtained by reordering the segments of $p$ in increasing order of slope. Repeating this algorithm will produce paths $p'$ that are lower and lower down, but still bounded by $p^\sharp$ from below (since the convexification $p^\sharp$ lies below of any convexifications $(p')^\sharp$). The fact that the area between these paths and $p^\sharp$ decreases by a positive integer multiple of $\frac 12$ at every step means that the algorithm must terminate after finitely many steps.
\end{proof}


\medskip

\subsection{The completion}
\label{sub:completion}

We are now ready to define our completion of $\UUm$.

\medskip

\begin{definition}
\label{def:completion}
Consider the vector space
\begin{equation}
\label{eqn:completion}
  \wUUm \ = \bigoplus_{(\bk,d) \in \nn \times \BZ} \wUUm_{-\bk,-d}
\end{equation}
where
$$
  \wUUm_{-\bk,-d} \ = \prod^{\text{convex path }p}_{\text{of size } (|\bk|,d)} \UUm_p
$$
\end{definition}

\medskip

\begin{proposition}
\label{prop:completion is an algebra}
The algebra structure on $\UUm$ extends uniquely to an algebra structure on $\wUUm$.
\end{proposition}

\medskip

\begin{proof}
We start by remarking that for any given path $p$ of size $(|\bk|,d)$, all but finitely many convex paths of the same size will lie below $p$. Let us argue by contradiction: suppose that infinitely many convex paths of size $(|\bk|,d)$ failed to lie below $p$. Then each of these infinitely many paths would need to pass through at least one of the finitely many lattice points contained inside the trapezoid bounded by the line segment $(0,0)$ to $(|\bk|,d)$, the two vertical lines with $x$-coordinate $0$ and $|\bk|$, and the horizontal line corresponding to the smallest $y$-coordinate of any point on the path $p$. It is clearly impossible for infinitely many convex paths of fixed size to pass through a fixed lattice point (other than the endpoints of the path).

\medskip
\noindent
With this in mind, our task is to show the well-definedness of a product
\begin{equation}
\label{eqn:the product}
  (\underbrace{\pi_1+\pi_2+\pi_3+\dots}_{\text{elements of }\UUm_{-\bk,-d}}) \cdot (\underbrace{\pi_1'+\pi_2'+\pi_3'+\dots}_{\text{elements of } \UUm_{-\bk',-d'}})
\end{equation}
where $\pi_s$ (respectively $\pi'_t$) correspond to different convex paths $p_s$ of size $(|\bk|,d)$ (respectively $p'_t$ of size $(|\bk'|,d')$). We may evaluate the product \eqref{eqn:the product} by foiling out the brackets and expressing each $\pi_s \pi_t'$ as a linear combination of products
\begin{equation}
\label{eqn:eph}
  \widetilde{\pi} \text{ as in \eqref{eqn:product p}}
\end{equation}
(cf.~\eqref{eqn:commute}). In order for such an infinite sum to be a well-defined element of $\wUUm$, one needs to show that each fixed $\widetilde{\pi}$ only appears for finitely many $(s,t)$. Let us analyze the possible paths $\widetilde{p} = \text{path}(\widetilde{\pi})$ that correspond to elements \eqref{eqn:eph}, and it suffices to show that any such path appears only for finitely many pairs $(s,t)$.

\medskip
\noindent
By Proposition \ref{prop:olivier}, any path $\widetilde{p}$ corresponding to a product $\widetilde{\pi}$ in \eqref{eqn:eph} must lie below the concatenation of $p_s$ with $p_t'$. However, by the argument in the first paragraph of the proof, for any $N \in \BN$ all but finitely many paths $p_s$ (respectively all but finitely many paths $p_t'$) contain some point with vertical coordinate $<-N$. Therefore, the same is true for the concatenation of $p_s$ and $p_t'$. By choosing $N$ large enough compared to any given path $\widetilde{p}$, we can ensure that $\widetilde{p}$ does not lie below the concatenation of $p_s$ with $p_t'$, except for finitely many pairs $(s,t)$.
\end{proof}

\medskip
\noindent
The proof of Proposition \ref{prop:completion is an algebra} also proves the following stronger statement.

\medskip

\begin{proposition}
\label{prop:completion strong}
For any $(\bk,d) \in \nn \times \BZ$, any countable sum of products
$$
  a_1 \dots a_t, \quad \text{where } a_s \in \wUUm_{-\bk_s,-d_s}, \forall s \in \{1,\dots,t\}, \sum_{s=1}^t \bk_s = \bk, \sum_{s=1}^t d_s = d
$$
also lies in $\wUUm$ as long as all but finitely many such products have the property that the (not necessarily convex) path with legs $(|\bk_1|,d_1), \dots, (|\bk_t|,d_t)$ lies below any given convex path of size $(|\bk|,d)$ in $\BN \times \BZ$.
\end{proposition}


\medskip

\subsection{Quantum $\widehat{\fg}$}

We will now recall another completion of quantum affine algebras, that we will shortly connect with $\wUUm$ above. Let $\UUext$ be the standard central extension of $\UU$ with an extra central generator $C$ such that
\begin{equation*}
  \UU \simeq \UUext/(C-1)
\end{equation*}
We will not need the precise definition of the central extension, but observe that it does not affect the subalgebra $\UUm$, or its completion \eqref{eqn:completion}.
The algebra $\UUext$ admits an automorphism $\varpi$ defined on the generators via
\begin{equation}
\label{eq:varpi-autom}
  \varpi\colon e_{i,d}\mapsto f_{i,-d},\quad f_{i,d} \mapsto e_{i,-d},\quad \ph^\pm_{i,d'} \mapsto \ph^\mp_{i,d'}, \quad C\mapsto C^{-1}
\end{equation}
for all $i\in I, d\in \BZ, d'\in \BN$.
Furthermore, let $\UUaff$ denote the Drinfeld--Jimbo quantum group associated to the affinization of the Lie algebra $\fg$, which is generated~by
$$
  \big\{e_i,f_i,\ph^{\pm 1}_i\big\}_{i\in \wI} \quad \mathrm{with} \quad \wI=I\sqcup 0
$$
A result of Drinfeld (proved by Beck, Damiani) establishes an algebra isomorphism
\begin{equation}
\label{eqn:iso beck}
  \Phi\colon \UUext \iso \UUaff
\end{equation}
Let $\UUaffp$, $\UUaffm$, $\UUaffh$ denote the subalgebras of $\UUaff$ generated by $e_i$, $f_i$, $\ph^\pm_i$, respectively. Let $\UUaffgl$ denote the subalgebras generated by $\UUaffpm$ and $\UUaffh$. According to~\cite{B} (respectively~\cite{KT2}), the subalgebras $\UUaffpm$ admit PBW bases in the root vectors $\{e_{\pm \gamma}\}_{\gamma\in \wDelta^+}$\footnote{For imaginary roots $\gamma$, we actually have $|I|$ root vectors $\{e^{(i)}_\gamma\}_{i\in I}$ instead of a single $e_\gamma$.} (termed \emph{Cartan--Weyl basis} in~\cite{KT2}) constructed through Lusztig's braid group action (respectively, via iterated $q$-commutators in~\cite{KT2}, for every normal ordering of $\wDelta^+$) where $\wDelta^+$ denotes the set of positive affine roots.

\medskip

\begin{claim}\label{claim:deep}
The root generators $e_\gamma$ with $\gamma\in \wDelta^+$ (respectively $-\gamma\in \wDelta^+$) can be expressed as non-commutative polynomials in $\{e_i\}_{i\in \wI}$ (respectively $\{f_i\}_{i\in \wI}$) of degree $\height(\gamma)$, with the height of an affine root defined by $\height(\sum_{i\in \wI} r_i\alpha_i)=\sum_{i\in \wI} r_i$.
\end{claim}

\medskip
\noindent
Consider the $\BZ$-grading of $\UUaffpm$ given by $\height$, i.e.\ setting $\deg e_i = 1$ and $\deg f_i = -1$ for all $i \in \wI$. With this in mind, we define the completion
\begin{equation}
\label{eqn:second completion}
  \wUUaff
\end{equation}
where a basis of neighborhoods of the identity are images under multiplication of
\begin{equation}
\label{eqn:neighborhoods of zero}
\UUaffp_{\geq N} \otimes \UUaffh \otimes \UUaffm_{\leq -N}
\end{equation}
as $N$ ranges over the natural numbers. Therefore, elements in the completion \eqref{eqn:second completion} are infinite sums of products, all but finitely many of which are of the form
$$
  \Big(\text{at least }N \text{ factors } e_i  \Big)_{i \in \wI} \cdot \Big(\text{anything}\Big) \cdot \Big(\text{at least } N \text{ factors }  f_i \Big)_{i \in \wI}
$$

\medskip

\begin{remark}\label{rem:DK-completion}
This completion is closely related to the one considered in~\cite{DK2}, which was defined to consist of infinite sums
$$
  \UUaffh \cdot e_{-\gamma}^{n_\gamma} \dots e_{-\beta}^{n_\beta} e_{-\alpha}^{n_\alpha} e_{\alpha}^{m_\alpha} e_{\beta}^{m_\beta} \dots e_{\gamma}^{m_\gamma}
$$
where $\alpha<\beta<\dots<\gamma$ with respect to the above normal ordering, such that for any weight $\lambda$ and any $N\in \BZ$, there are only finitely many terms of total weight $\lambda$ which satisfy the condition $(n_\alpha+m_\alpha)\height(\alpha) + (n_\beta+m_\beta)\height(\beta) + \dots + (n_\gamma+m_\gamma)\height(\gamma) \leq N$.
\end{remark}


\medskip

\subsection{The two completions}

We will now connect the completion $\wUUm$ of the subalgebra $\UUm$ of the left-hand side of \eqref{eqn:iso beck} with the completion \eqref{eqn:second completion} of the right-hand side. Our main result on this matter is the following.

\medskip

\begin{proposition}
\label{prop:completions compatible}
The map \eqref{eqn:iso beck} induces an injective algebra homomorphism
\begin{equation}
\label{eqn:completions compatible}
  \Phi\colon \wUUm \longrightarrow \wUUaff
\end{equation}
\end{proposition}

\medskip

\begin{proof}
We recall the fact (see~\cite{B2,Da}) that the isomorphism \eqref{eqn:iso beck} satisfies
$$
  \Phi(f_{i,d})\in
  \begin{cases}
    \UUaffm_{\leq -1} &\text{if }d \leq 0 \\
    \UUaffg_{\geq 1} &\text{if }d > 0
  \end{cases}
$$
In order to conclude the existence of a homomorphism \eqref{eqn:completions compatible}, we need to show that infinite sums of \eqref{eqn:product p} corresponding to convex paths lie in the completion \eqref{eqn:second completion}. To this end, let us fix a degree $\bk \in \nn$, and consider the \underline{finitely many} subspaces
\begin{equation}
\label{eqn:the subspaces}
  \CB_{-\bk',-d'} \neq 0
\end{equation}
with $0 \leq d' < |\bk'|$ and $\bk' \leq \bk$. All the subspaces \eqref{eqn:the subspaces} are finite-dimensional, so we may conclude that all their elements are written as sums of products of $f_{i,d}$'s with $d \in [-M,M]$ for some large enough natural number $M = M(\bk)$. Because of this, and the fact that the shift automorphism $\{f_{i,d} \mapsto f_{i,d-1}\}_{i \in I}^{d\in \BZ}$ sends $\CB_{-\bk',-d'} \iso \CB_{-\bk',-d'-|\bk'|}$, we have
$$
  \CB_{-\bk',-d'} \subset
  \begin{cases}
    \Big \langle f_{i,0},f_{i,-1},\dots \Big \rangle_{i \in I}  \subseteq \UUaffm &\text{if }  d' \geq M|\bk'| \\
    \Big \langle f_{i,1},f_{i,2},\dots \Big \rangle_{i \in I} \subseteq \UUaffg &\text{if } d' < -M|\bk'|
  \end{cases}
$$
Since the degree by height of $\Phi(\CB_{-\bk',-d'})\subset \UUaff$ is equal to $-|\bk'|-d'\ell$, where $\ell \in \BZ_{>0}$ denotes the height of the minimal imaginary root $\delta$, and $\bk' \leq \bk$, we conclude that the subalgebras $\CB_{-\bk',-d'}$ correspond to elements of $\UUaff$ as follows
$$
  \begin{cases}
    \text{of degree}\leq O(-\mu) &\text{if }\mu\geq M\\
    \text{of bounded degree } & \text{if } \mu \in [-M ,M) \\
    \text{of degree} \geq O(-\mu) &\text{if } \mu < - M
  \end{cases}
$$
where $\mu = \frac {d'}{|\bk'|}$ (above and below, ``$O(\mu)$" refers to $a\mu+b$ for some $a>0$, $b \in \BR$). For any $N \in \BN$, we may subdivide a convex path into ``segments of slope $\leq O(-N)$", ``segments of intermediate slope" and ``segments of slope $\geq O(N)$". In doing so, we can ensure that any product \eqref{eqn:product p} give rise to an element of the form \eqref{eqn:neighborhoods of zero} in $\UUaff$. This implies both that the map $\Phi$ is well-defined and that it is injective, because the slope (which measures the depth of the neighborhoods of 0 in the completion \eqref{eqn:completion}) is linear in the number $N$ (which measures the depth of the neighborhoods of zero in the completion \eqref{eqn:second completion}).
\end{proof}

\medskip
\noindent
Although it is inconsequential for the remainder of the present paper, we make the following conjecture, which will be proved in~\cite{N24}.

\medskip

\begin{conjecture}
\label{conj:completions compatible}
The map $\Phi$ of~\eqref{eqn:iso beck} sends the slope subalgebras as follows
\begin{align*}
  & \Phi(\CB_{\mu}^+) \subset \UUaffp , \qquad \text{if } \mu \geq 0 \\
  & \Phi(\CB_{\mu}^+) \subset \UUaffl , \qquad \text{if } \mu < 0  \\
  & \Phi(\CB_{\mu}^-) \subset \UUaffm , \qquad \text{if } \mu \leq 0 \\
  & \Phi(\CB_{\mu}^-) \subset \UUaffg , \qquad \text{if } \mu > 0
\end{align*}
\end{conjecture}


\medskip

\subsection{The braid group action}

Although they do not use the language of completions (but point out the correct completion in their other paper~\cite{DK2}, see Remark~\ref{rem:DK-completion}) and instead work in a certain class of admissible representations,~\cite{DK} defined an action of the braid group of type $\fg$, generated by $\{T_i\}_{i \in I}$, on the completion \eqref{eqn:second completion}.
The authors of~\cite{DK} work only with the simply-laced $\fg$, in which case the explicit formulas for $T_i$ can be read off from~\cite[(1.15)--(1.21), (2.1)--(2.5)]{DK}:
\begin{itemize}

\item[$\bullet$]
if $a_{ij}=0$, then for any $d\in \BZ, d'\in \BN$
\begin{equation}\label{eq:non-adjacent}
  T_i(e_{j,d})=e_{j,d}, \qquad T_i(f_{j,d})=f_{j,d}, \qquad T_i(\ph^\pm_{j,d'})=\ph^\pm _{j,d'}
\end{equation}

\item[$\bullet$]
if $a_{ij}=-1$, then for any $d\in \BZ, d'\in \BN$
\begin{equation}\label{eq:adjacent}
\begin{split}
  & T_i(e_{j,d}) = q^{2d} \left( -e_{j,d}e_{i,0} + q^{-1}e_{i,0}e_{j,d} - (q-q^{-1})\sum_{k\geq 1} q^{-k} e_{i,k} e_{j,d-k} \right) , \\
  & T_i(f_{j,d}) = q^{2d} \left( f_{i,0}f_{j,d} - qf_{j,d}f_{i,0} - (q-q^{-1})\sum_{k\geq 1} q^{k} f_{j,d+k}f_{i,-k} \right) , \\
  & T_i(\ph^\pm_{j,d'}) = \sum_{k=0}^{d'} q^{\mp(k-2d')} \ph^\pm_{i,k} \ph^\pm_{j,d'-k}
\end{split}
\end{equation}

\item[$\bullet$]
if $j=i$, then for any $d\in \BZ, d'\in \BN$
\begin{equation}\label{eq:same-vertices}
\begin{split}
  & T_i(e_{i,d}) = \sum_{k\geq 0} \bph^+_{i,k}f_{i,d-k}C^{2d-k}, \\
  & T_i(f_{i,d}) = \sum_{k\geq 0} e_{i,d+k}\bph^-_{i,k}C^{2d+k}, \\
  & T_i(\ph^\pm_{i,d'}) = \bph^\pm_{i,d'},
\end{split}
\end{equation}
where $\{\bph^\pm_{i,k}\}_{i\in I}^{k\geq 0}$ are defined so that
$\bph^\pm_i(z)=\sum_{k\geq 0} \bph^\pm_{i,k} z^{\mp k} = \left(\ph^\pm_i(z)\right)^{-1}$.

\end{itemize}
where we applied the automorphism $\varpi$ of~\eqref{eq:varpi-autom}, as our $e_{i,d}, f_{i,d}$ are $f_{i,-d}, e_{i,-d}$ of \emph{loc.~cit.}
With this in mind, we propose the following reinterpretation of the main result of~\cite{DK}.

\medskip

\begin{proposition}
The maps $T_i$ of~(\ref{eq:non-adjacent})--(\ref{eq:same-vertices}) give rise to well-defined automorphisms
$$
  \wUUaff \iso \wUUaff
$$
which induce an action of the braid group on $\wUUaff$.
\end{proposition}


\medskip

\subsection{Fused currents}
\label{sub:fused}

Recall that any reduced decomposition of the longest element in the Weyl group associated to~$\fg$ (with $s_i=s_{\alpha_i}$ being the simple reflections)
\begin{equation}
\label{eqn:reduced decomposition}
  w_0 = s_{i_1} \dots s_{i_n}
\end{equation}
yields a total (convex) order on the set $\Delta^+$ of positive roots of $\fg$, by setting
\begin{equation}
\label{eqn:total order}
  s_{i_n} \dots s_{i_2}(\alpha_{i_1}) > s_{i_n} \dots s_{i_3}(\alpha_{i_2}) >  \dots > s_{i_n}(\alpha_{i_{n-1}}) > \alpha_{i_n}
\end{equation}

\medskip

\begin{definition}
\label{def:fused currents}
(\cite{DK}) For any reduced decomposition \eqref{eqn:reduced decomposition}, any positive root $\alpha = s_{i_n} \dots s_{i_{k+1}}(\alpha_{i_k})$ (for some $k \in \{1,\dots,n\}$), and any $d \in \BZ$, let
\begin{equation}
\label{eqn:fused general}
  \tf_{\alpha,d} = T^{-1}_{i_n} \dots T^{-1}_{i_{k+1}}(f_{i_k,d}) \in \wUUaff
\end{equation}
The fused current is defined by $\tf_{\alpha}(x) = \sum_{d \in \BZ} \frac {\tf_{\alpha,d}}{x^d}$.
\end{definition}

\medskip
\noindent
The original definition of \cite{DK} had the expressions \eqref{eqn:fused general} defined in an appropriate class of representations in which certain countable sums are well-defined. We posit that their definition is equivalent to the completion of Definition \ref{def:completion}.

\medskip

\begin{conjecture}
\label{conj:dk}
For any $(\alpha,d) \in \Delta^+ \times \BZ$, we have
$$
  \tf_{\alpha,d} \in \emph{Im}\Big( \wUUm_{-\alpha,d} \longrightarrow \wUUaff \Big)
$$
with respect to the algebra homomorphism~\eqref{eqn:completions compatible} of Proposition \ref{prop:completions compatible}.
\end{conjecture}

\medskip
\noindent
We expect that the leading term of $\tf_{\alpha,0}$ is precisely the root vector $f_{\alpha} \in \uum$ of~\eqref{eqn:root vectors} under the natural embedding $\uum\hookrightarrow \UUm$ (that is, $f_{\alpha,0} = f_{\alpha}$ plus products \eqref{eqn:product p} whose associated convex paths lie strictly below the horizontal line). Therefore, it is natural to develop the commutation relations for fused currents, akin to Levendorskii--Soibelman formulas \eqref{eqn:commutation intro} in the finite type quantum group. In particular, we will focus on the case of a minimal pair of positive roots $\alpha < \beta$, and seek an analogue of \eqref{eqn:commutation intro minimal}.

\medskip

\begin{conjecture}
\label{conj:comm}
For any minimal pair $\alpha < \beta$ w.r.t.\ the order~\eqref{eqn:total order}, we have
\begin{equation}
\label{eqn:conj comm}
  \tf_{\alpha}(x) \tf_{\beta}(y) - \tf_{\beta}(y) \tf_{\alpha}(x) \cdot \frac {xq^{\tt} - yq^{(\alpha,\beta)}}{xq^{\tt+(\alpha,\beta)}-y} =
  c(q) \cdot \delta\left( \frac {xq^{\tt+(\alpha,\beta)}}y \right) \tf_{\alpha+\beta}(xq^u)
\end{equation}
for some $u,\tt \in \BZ$ and $c(q) \in \BZ[q,q^{-1}]^\times$ depending on $\alpha,\beta$, where $\delta(x) = \sum_{d \in \BZ} x^d$. In the second term from the LHS of \eqref{eqn:conj comm}, the rational function is expanded in the region $|x|\gg |y|$.
\end{conjecture}

\medskip
\noindent
Note that Conjecture \ref{conj:comm} implies Conjecture \ref{conj:dk}, by the following inductive argument. Formula \eqref{eqn:conj comm} allows us to write for all $(\alpha,d) \in \Delta^+ \times \BZ$ (with $\height(\alpha)>1$)
\begin{equation*}
  \tf_{\alpha,d} = \text{constant} \cdot \tf_{\beta,d} \tf_{\gamma,0} + \sum_{n=0}^{\infty} \text{constant} \cdot \tf_{\gamma,n} \tf_{\beta, d-n}
\end{equation*}
for some $\beta,\gamma \in \Delta^+$ with $\alpha = \beta+\gamma$. If all the $\tf$'s in the right-hand side lie in $\wUUm$ by the induction hypothesis, then Proposition \ref{prop:completion strong} implies that so does the entire right-hand side. Therefore, so does the left-hand side, which establishes the induction step.


\medskip

\subsection{Pairing with elements in the completion}

We will now explain the importance of Conjecture \ref{conj:dk}, i.e.\ why is it important that $\tf_{\alpha,d}$ should be interpreted as lying in the completion \eqref{eqn:completion} rather than the completion \eqref{eqn:second completion}. To us, this is relevant because the former completion interacts with the pairing \eqref{eqn:pairing} in the following natural way.

\medskip

\begin{proposition}
\label{prop:completion pairing}
The pairing \eqref{eqn:pairing} naturally extends to a pairing
\begin{equation}
\label{eqn:extended pairing}
  \CS \otimes \wUUm \xrightarrow{\langle \cdot, \cdot \rangle} \BQ(q)
\end{equation}
For any $(\bk,d) \in \nn \times \BZ$, any linear functional $\CS_{\bk,d} \xrightarrow{\lambda} \BQ(q)$ can be written as
\begin{equation}
\label{eqn:required for linear functional}
  \lambda(R) = \langle R, f \rangle, \qquad \forall\, R\in \CS_{\bk,d}
\end{equation}
for a unique $f \in \wUUm_{-\bk,-d}$.
\end{proposition}

\medskip

\begin{proof}
Let us recall the decomposition \eqref{eqn:factorization 1} for $\CS$ and \eqref{eqn:factorization 3} for $\UUm$. Any given $R \in \CS$ decomposes in terms of the slope subalgebras $\CB_\mu^+$ for $\mu$ in a finite subset $\Omega \subset \BQ$. Proposition \ref{prop:pair slopes} says that the pairing \eqref{eqn:pairing} respects these decompositions, and so $R$ pairs non-trivially only with products \eqref{eqn:product p} whose corresponding convex path has legs with slopes in $\Omega$. But only finitely many convex paths of any given size have legs with slopes in the finite subset $\Omega \subset \BQ$, so we conclude that the pairing of $R$ with any countable sum of products making up $\wUUm$ is well-defined.

\medskip
\noindent
Let us now prove the statement about $\lambda$. Recall that the restriction of \eqref{eqn:pairing} to
\begin{equation*}
  \CB_\mu^+ \otimes \CB_\mu^- \xrightarrow{\langle \cdot, \cdot \rangle} \BQ(q)
\end{equation*}
is a non-degenerate pairing of graded vector spaces which are finite-dimensional in every degree (and thus a perfect pairing). Therefore, so is
\begin{equation}
\label{eqn:pair restricted 2}
  \Big( \CB_{\bk_1,d_1} \otimes \dots \otimes \CB_{\bk_t,d_t} \Big) \otimes \Big( \CB_{-\bk_1,-d_1} \otimes \dots \otimes \CB_{-\bk_t,-d_t} \Big) \xrightarrow{\langle \cdot, \cdot \rangle} \BQ(q)
\end{equation}
for any convex path with legs $(|\bk_1|,d_1), \dots, (|\bk_t|,d_t)$ in $\BN \times \BZ$ of size $(|\bk|,d)$, in virtue of Proposition \ref{prop:pair slopes}. For such a convex path, let us denote
$$
  \lambda_{(\bk_1,d_1),\dots,(\bk_t,d_t)}\colon \underbrace{\CB_{\bk_1,d_1} \otimes \dots \otimes \CB_{\bk_t,d_t}}_{\text{a direct summand of }\CS} \rightarrow \BQ(q)
$$
the appropriate restriction of $\lambda$. Since the pairing \eqref{eqn:pair restricted 2} is perfect, there exists
$$
  f_{(\bk_1,d_1),\dots,(\bk_t,d_t)} \in \underbrace{\CB_{-\bk_1,-d_1} \otimes \dots \otimes \CB_{-\bk_t,-d_t}}_{\text{a direct summand of } \UUm}
$$
such that $\lambda_{(\bk_1,d_1),\dots,(\bk_t,d_t)}$ is given by pairing with $f_{(\bk_1,d_1),\dots,(\bk_t,d_t)}$. Then letting
$$
  f = \mathop{\sum_{\text{convex paths of size } (|\bk|,d)}}_{\text{with legs } (|\bk_1|,d_1),\dots,(|\bk_t|,d_t)} f_{(\bk_1,d_1),\dots,(\bk_t,d_t)}
$$
yields the required element of $\wUUm$ in \eqref{eqn:required for linear functional}. The uniqueness of such $f$ satisfying \eqref{eqn:required for linear functional} is due to non-degeneracy of \eqref{eqn:pairing}.
\end{proof}


\medskip

\subsection{Specialization maps}
\label{sub:specialization}

If the coefficients of the fused currents $\tf_\alpha(x)$ lie in the completion $\wUUm$, as predicted by Conjecture \ref{conj:dk}, then Proposition \ref{prop:completion pairing} implies that they have a well-defined pairing with elements of the shuffle algebra.

\medskip

\begin{definition}
\label{def:spec}
For any positive root $\alpha$, define the specialization map
\begin{equation}
\label{eqn:specialization}
  \CS_{\alpha} \xrightarrow{\tespec^{(x)}_{\alpha}} \BQ(q)[x,x^{-1}], \qquad \tespec^{(x)}_\alpha(R) = \Big \langle R, \tf_{\alpha}(x) \Big \rangle
\end{equation}
\end{definition}

\medskip
\noindent
Such specialization maps were studied in (super)type $A$ in \cite{Ts, Ts0}, in types $B_n$ and $G_2$ in \cite{HT}, and in affine type $A$ in \cite{N0}, for a specific choice of the order on $\Delta^+$. In all of these cases, the specialization maps were given by setting
\begin{equation}
\label{eqn:collection}
  \tspec^{(x)}_\alpha(R) = \widetilde{\gamma}_\alpha^{(x)} \cdot R(\dots,z_{ib},\dots)\Big|_{z_{ib} \mapsto xq^{\sigma_{ib}}, \forall i,b}
\end{equation}
for some collection of integers $(\sigma_{ib})_{i \in I, b \geq 1}$ and some prefactor $\widetilde{\gamma}^{(x)}_{\alpha} \in \BQ(q)^\times \cdot x^\BZ$. In general, we expect the specialization maps to be given by a suitable derivative of $R$ evaluated at a collection as in the right-hand side of \eqref{eqn:collection}, up to a prefactor.

\medskip
\noindent
It would be very interesting to obtain a complete description of the specialization maps \eqref{eqn:specialization} for any finite type root system and any reduced decomposition \eqref{eqn:reduced decomposition} of the longest word; in Section \ref{sec:quivers}, we provide such a description in the particular case of ADE type quivers. However, we emphasize the fact that the specialization maps $\tspec^{(x)}_{\alpha}$ should be considered in relation to the conjectural formula \eqref{eqn:conj comm}. Specifically, if $\alpha < \beta$ is a minimal pair and $R \in \CS_{\alpha+\beta}$, then we have
\begin{equation}
\label{eqn:expansion 1}
  \Big \langle R, \tf_{\alpha}(x) \tf_{\beta}(y) \Big \rangle =
  \frac {\tspec_\alpha^{(x)} \otimes \tspec_{\beta}^{(y)}(R)}{\prod_{i,b} \prod_{j,c} \zeta_{ij} \left(\frac {xq^{\sigma_{ib}}}{yq^{\tau_{jc}}}\right)} \Big|_{\text{expanded as } |x| \ll |y|}
\end{equation}
(this formula comes from a topological coproduct on the Cartan-extended version of $\CS$, cf.~\cite[(2.35)]{N1}), where $(\sigma_{ib})$ and $(\tau_{jc})$ are the collections of integers associated to the specialization maps $\tspec_{\alpha}^{(x)}$ and $\tspec_{\beta}^{(y)}$, respectively. Here, $\tspec_\alpha^{(x)} \otimes \tspec_{\beta}^{(y)}(R)$ means that one divides the variables of $R$ into two groups: the variables in one of the groups are specialized according to $\tspec_{\alpha}^{(x)}$, and the variables in the other group are specialized according to $\tspec_{\beta}^{(y)}$. Meanwhile, we postulate that
\begin{equation}
\label{eqn:rational function}
  \prod_{i,b} \prod_{j,c} \frac {\zeta_{ji} \left(\frac {yq^{\tau_{jc}}}{xq^{\sigma_{ib}}}\right)}{\zeta_{ij} \left(\frac {xq^{\sigma_{ib}}}{yq^{\tau_{jc}}}\right)} =
  \frac {xq^{\tt} - yq^{(\alpha,\beta)}}{xq^{\tt+(\alpha,\beta)}-y}
\end{equation}
with $\tt$ as in~\eqref{eqn:conj comm}, so that we have
\begin{equation}
\label{eqn:expansion 2}
  \left \langle R,  \tf_{\beta}(y) \tf_{\alpha}(x) \frac {xq^{\tt} - yq^{(\alpha,\beta)}}{xq^{\tt+(\alpha,\beta)}-y} \right \rangle =
  \frac {\tspec_\alpha^{(x)} \otimes \tspec_{\beta}^{(y)}(R)}{\prod_{i,b} \prod_{j,c} \zeta_{ij} \left(\frac {xq^{\sigma_{ib}}}{yq^{\tau_{jc}}}\right)} \Big|_{\text{expanded as } |x| \gg |y|}
\end{equation}
Comparing the right-hand sides of the expansions \eqref{eqn:expansion 1} and \eqref{eqn:expansion 2}, we see that we have the exact same rational function in $x$ and $y$, first expanded as $|x| \ll |y|$ and then expanded as $|x| \gg |y|$. Therefore, formula \eqref{eqn:conj comm} precisely entails the fact that said rational function has a single pole at $y = xq^{\tt+(\alpha,\beta)}$, and that the corresponding residue at this pole is none other than $c(q)\cdot \tspec_{\alpha+\beta}^{(xq^u)}(R)$. We conclude that specialization maps give us a novel (and dual) way of thinking about the conjectural commutation relation \eqref{eqn:conj comm} of fused currents.


\bigskip

\section{Quivers of ADE type}
\label{sec:quivers}


\medskip

\subsection{Quivers and Hall algebras}

We will henceforth assume that $\fg$ is a simply-laced finite type Lie algebra, and we choose an orientation $Q$ of the Dynkin diagram of $\fg$. The inner product $(\cdot, \cdot)$ on the root lattice satisfies $(\beta,\alpha)\in \{-1,0,1\}$ for any $\alpha,\beta\in \Delta$ with $\beta\ne \pm \alpha$. Having made this choice, we may consider the pairing
\begin{equation}
\label{eqn:lattice pairing}
  \langle \cdot, \cdot \rangle \colon  \BZ^I \times \BZ^I \longrightarrow \BZ, \qquad \langle \bv, \bw \rangle = \sum_{i \in I} v_i w_i - \sum_{\overrightarrow{ij}} v_i w_j
\end{equation}
and note that $(\bv,\bw) = \langle \bv, \bw \rangle + \langle \bw, \bv \rangle$. Let $\BF_{q^2}$ be a finite field.

\medskip

\begin{definition}
\label{def:hall}
Consider the category $\mathcal{C}$ of finite-dimensional representations of the quiver $Q$, i.e.\ collections of finite-dimensional vector spaces over $\BF_{q^2}$ associated to the vertices of $Q$ and linear maps associated to the edges of $Q$
\begin{equation}\label{eqn:arrow map}
  V = \Big( V_i \xrightarrow{\phi_e} V_j \Big)_{i,j \in I, e = \oij}
\end{equation}
modulo change of basis of the vector spaces $V_i$. The Hall algebra of $\mathcal{C}$ is defined as
$$
  \Hall = \Hall(\mathcal{C}) \, = \bigoplus_{[V] \in \emph{Ob}(\mathcal{C})/\sim} \BQ \cdot [V]
$$
endowed with the multiplication
\begin{multline}
\label{eqn:hall}
  [V] \cdot [W] = q^{\langle \bv,\bw \rangle} \cdot \\ \sum_{[X] \in \emph{Ob}(\mathcal{C})/\sim} [X] \cdot \#\big \{\text{subreps } Y \subset X \text{ s.t. } Y \simeq W ,\ X/Y \simeq V \big\}
\end{multline}
where $\bv = (\dim V_i)_{i \in I}$ and $\bw = (\dim W_i)_{i \in I}$.
\end{definition}

\medskip
\noindent
It is well-known that the structure constants of the algebra $\Hall$ (i.e.\ the numbers that appear in the RHS of \eqref{eqn:hall}) are Laurent polynomials in $q$ with rational coefficients. Thus, one can think of $q$ as a formal parameter, and of $\Hall$ as an algebra over $\BQ(q)$. With this in mind, we have the following foundational result.

\medskip

\begin{theorem}(\cite{Gr,R7})
There is an algebra isomorphism
\begin{equation}
\label{eqn:iso hall}
  \uup \iso \Hall
\end{equation}
determined by sending the generator $e_i$ to the simple quiver representation with a one-dimensional vector space at the vertex $i$ and $0$ everywhere else.
\end{theorem}

\medskip
\noindent
We may grade $\Hall$ by associating to any quiver representation $V$ its dimension vector $\bv=(\text{dim }V_i)_{i \in I} \in \nn$, with respect to which \eqref{eqn:iso hall} becomes an isomorphism of $\nn$-graded algebras. For any $V,W\in \mathcal{C}$ with dimension vectors $\bv$ and $\bw$, respectively, the pairing $\langle \bv,\bw \rangle$ of~\eqref{eqn:lattice pairing} coincides with the Euler form (see~\cite{R1}):
$$
  \langle \bv,\bw \rangle = \sum_{k\geq 0} (-1)^k \dim \text{Ext}^k(V,W) = \dim\, \text{Hom}(V,W) - \dim \text{Ext}^1(V,W)
$$
with the second equality based on the vanishing
$$
  \text{Ext}^k(V,W)=0, \qquad \forall\, k\geq 2, \quad \forall\, V,W \in \text{Ob}(\mathcal{C})
$$


\medskip

\subsection{The Auslander--Reiten partial order}

It is well-known (\cite{Ga}) that indecomposable representations of $Q$ are in one-to-one correspondence with positive roots of $\fg$, i.e.\ up to isomorphism there is a single indecomposable representation $V_\alpha$ with dimension vector $\alpha\in \Delta^+$ (and there are no other indecomposables). For any two positive roots $\alpha$ and $\beta$, we have (\cite[Section 4]{Ri}) that
\begin{equation}
\label{eqn:Hom-or-Ext}
  \text{either} \quad \text{Hom}(V_\alpha,V_\beta) = 0 \quad \text{or} \quad \text{Ext}^1(V_\alpha,V_\beta) = 0
\end{equation}
As shown in \cite[Theorem 7]{Ri}, the isomorphism \eqref{eqn:iso hall} sends
$$
  e_{\alpha} \mapsto q^{\kappa_\alpha}\cdot [V_{\alpha}]
$$
for every positive root $\alpha$, where the root vectors $e_{\alpha}$ are defined with respect to a certain reduced decomposition of the longest word $w_0$ (see \cite[Section 13]{Ri} for how to construct this reduced decomposition starting from an orientation of the Dynkin diagram of $\fg$) and $\kappa_\alpha\in \BZ$ are not presently important to us (see~\cite[Section~3]{Ri}). Very interestingly, the behavior and commutation relations of the $e_{\alpha}$'s do not depend on the total order on $\Delta^+$ induced by the reduced decomposition, but rather only on the Auslander--Reiten partial order, which we will now recall.

\medskip

\begin{definition}
\label{def:ar}
The Auslander--Reiten quiver associated to $Q$ has the set $\Delta^+$ of positive roots as vertices, and has an arrow $\alpha \rightarrow \beta$ if
\begin{itemize}

\item
$\alpha \neq \beta$,

\item
$\langle \alpha, \beta \rangle > 0$,

\item
if $\gamma \in \Delta^+$ satisfies $\langle \alpha, \gamma \rangle > 0$ and $\langle \gamma,\beta \rangle>0$, then $\gamma \in \{\alpha,\beta\}$.

\end{itemize}
The Auslander--Reiten (AR for short) partial order on $\Delta^+$ is defined by the property that $\alpha > \beta$ if and only if there is a path from $\alpha$ to $\beta$ in the Auslander--Reiten quiver.
\end{definition}

\medskip
\noindent
The AR partial order has the following properties for all positive roots $\alpha \neq \beta$:
\begin{itemize}
	
\item
If $\alpha$ and $\beta$ are incomparable, then $\langle \alpha,\beta \rangle = 0$,
	
\item
If $\alpha < \beta$ then $\langle \beta,\alpha \rangle \geq 0 \geq  \langle \alpha, \beta \rangle$,
	
\item
If $\langle \alpha , \beta \rangle < 0$ or $\langle \beta , \alpha \rangle > 0$, then $\alpha < \beta$.

\end{itemize}
The first and the second properties follow from the third one, which the interested reader may find in \cite[Section~4]{Ri}. Moreover, the AR partial order is convex in the sense of \eqref{eqn:convex intro} (this claim is implicit in \cite[Section 13]{Ri}, who shows that any refinement of the AR partial order to a total order corresponds to a reduced decomposition of the longest word in the Weyl group; such total orders are known to be convex).


\medskip

\begin{lemma}
\label{lem:minimal ar}
If $\alpha < \beta$ is a minimal pair adding up to a positive root $\alpha+\beta$, then
\begin{equation}
\label{eqn:minimal ar}
  \langle \alpha,\beta\rangle = -1 \quad \text{and} \quad \langle \beta,\alpha\rangle = 0
\end{equation}
\end{lemma}

\medskip

\begin{proof}
Let us assume that the positive roots $\alpha$, $\beta$ and $\alpha+\beta$ correspond to indecomposable quiver representations $V$, $W$ and $X$, respectively.
Then, combining $\langle \alpha,\beta\rangle\leq 0\leq \langle \beta,\alpha \rangle$ with $-1=(\alpha,\beta)=\langle \alpha,\beta\rangle+\langle \beta,\alpha \rangle$ and~\eqref{eqn:Hom-or-Ext}, we get:
$$
  \dim \text{Ext}^1(V,W) = -\langle \alpha,\beta\rangle = t
$$
and
$$
  \dim \text{Hom}(W,V) = \langle \beta,\alpha\rangle = t-1
$$
for some $t \geq 1$. We will assume for the purpose of contradiction that $t \geq 2$, which implies that the space of extensions $\text{Ext}^1(V,W)$ is at least 2-dimensional. For any non-trivial extension
$$
  0 \rightarrow W \rightarrow S_1 \oplus \dots \oplus S_k \rightarrow V \rightarrow 0
$$
(for various indecomposables $S_1,\dots,S_k$) the fact that $V$ and $W$ are indecomposable implies that $\text{Hom}(W,S_a)$ and $\text{Hom}(S_a,V)$ are non-zero for all $1\leq a\leq k$. Therefore, the dimension vectors $\gamma_a$ of the $S_a$ are contained between $\alpha$ and $\beta$, due to the assumption that the extension is non-trivial as well as the vanishing of $\text{Ext}^1$ in~\eqref{eqn:Hom-or-Ext} and the third property of the AR partial order above. The following result will be proved at the end of the present proof.

\medskip

\begin{claim}
\label{claim:1}
If $\alpha< \beta$ is a minimal pair adding up to a positive root $\alpha+\beta$, there do not exist positive roots $\alpha < \gamma_1,\dots,\gamma_k < \beta$ with $k >1$ which add up to $\alpha+\beta$.
\end{claim}

\medskip
\noindent
As a consequence of Claim~\ref{claim:1}, we conclude that all non-zero elements in $\text{Ext}^1(V,W)$ are of the form
\begin{equation}
\label{eqn:ses 1}
  0 \rightarrow W \xrightarrow{f} X \xrightarrow{g} V \rightarrow 0
\end{equation}
Since we assumed that the space of extensions in question is at least 2-dimensional, let us consider another extension
\begin{equation}
\label{eqn:ses 2}
  0 \rightarrow W \xrightarrow{f'} X \xrightarrow{g'} V \rightarrow 0
\end{equation}
which is linearly independent from \eqref{eqn:ses 1}. The Baer sum of these two extensions
\begin{equation*}
  0 \rightarrow W \xrightarrow{f''} Q \xrightarrow{g''} V \rightarrow 0
\end{equation*}
which is defined by setting $Q = \text{Ker}(g,g')/\text{Im}(f,f')$ with
$$
  W \xrightarrow{(f,f')} X \oplus X \xrightarrow{(g,g')} V
$$
is nonzero. By Claim \ref{claim:1} and the sentence immediately following it, we therefore have $Q \simeq X$. Consider the quiver representation $\text{Ker}(g,g')=T_1\oplus \dots \oplus T_l$ (for various indecomposables $T_a$'s) which has dimension $\alpha+2\beta$. Invoking \eqref{eqn:Hom-or-Ext} and the third property of the AR partial order stated after Definition \ref{def:ar}, we conclude that $\dim X\leq \dim T_a\leq \dim W$ for all $1\leq a\leq l$, because $\text{Hom}(T_a,X)\ne 0$ and either $\text{Hom}(W,T_a)\ne 0$ or $T_a$ is an indecomposable summand in $Q\simeq X$ so that $T_a\simeq X$. We can thus apply the following analogue of Claim \ref{claim:1}, which will be proved several paragraphs down.

\medskip

\begin{claim}
\label{claim:2}
If $\alpha< \beta$ is a minimal pair adding up to a positive root $\alpha+\beta$, there do not exist positive roots $\alpha +\beta < \gamma_1,\dots,\gamma_k < \beta$ (respectively $\alpha < \gamma_1,\dots,\gamma_k < \alpha+\beta$) which add up to $\alpha+2\beta$ (respectively $2\alpha+\beta$).
\end{claim}

\medskip
\noindent
Therefore and with the convexity in mind, we have $\text{Ker}(g,g')\simeq W \oplus X$, so we must have a short exact sequence
$$
  0 \rightarrow W \oplus X \rightarrow X \oplus X \xrightarrow{(g,g')} V \rightarrow 0
$$
Since the only endomorphisms of $X$ are scalars (this is true for all indecomposable representations, due to~\eqref{eqn:Hom-or-Ext} and the equality $\langle \gamma,\gamma\rangle = \tfrac{1}{2} (\gamma,\gamma) = 1$ for any positive root $\gamma$), we conclude that $g$ and $g'$ are scalar multiples of each other. Similarly, one proves that $f$ and $f'$ are scalar multiples of each other. But this contradicts the fact that the short exact sequences \eqref{eqn:ses 1} and \eqref{eqn:ses 2} are linearly independent.

\medskip
\noindent
Let us now prove Claim \ref{claim:1}. Assume for the purpose of contradiction that such $\gamma_1,\dots,\gamma_k$ existed, and choose a minimal $k \geq 2$ with said property. If $k = 2$, then this violates the minimality of the pair $\alpha < \beta$, hence $k \geq 3$. However, the fact that
$$
  2 = (\alpha+\beta,\alpha+\beta) = (\gamma_1+\dots+\gamma_k,\gamma_1+\dots +\gamma_k) = 2k + 2\sum_{a<b} (\gamma_a,\gamma_b)
$$
implies that there must exist $a < b$ with $(\gamma_a,\gamma_b) = -1$. Therefore, $\gamma_a + \gamma_b$ is a positive root, which contradicts the minimality of $k$.

\medskip
\noindent
Let us now prove Claim \ref{claim:2}. Assume for the purpose of contradiction that there exist positive roots
$$
  \alpha+\beta < \gamma_1 ,\dots,\gamma_k < \beta
$$
which add up to $\alpha+2\beta$ (the non-existence of $\alpha < \gamma_1 ,\dots,\gamma_k < \alpha+\beta$ which add up to $2\alpha+\beta$ is analogous, and is left as an exercise to the reader). Let us assume that $k$ is minimal with this property. Since $\alpha$, $\beta$ and $\alpha+\beta$ are positive roots, $\alpha+2\beta$ is not a positive root (as we cannot have all four of them having the same length), hence $k\geq 2$. If $k = 2$, then we have
$$
  (\gamma_1+\gamma_2,\beta) = (\alpha+2\beta,\beta) = 3
$$
which is impossible because $(\gamma,\beta) \leq 1$ for every positive root $\gamma \neq \beta$. Therefore, $k \geq 3$ and we have
$$
  6 = (\alpha+2\beta,\alpha+2\beta) = (\gamma_1+\dots+\gamma_k,\gamma_1+\dots +\gamma_k) = 2k + 2\sum_{a<b} (\gamma_a,\gamma_b)
$$
If $k > 3$, then there would exist $a < b$ with $(\gamma_a,\gamma_b) = -1$, hence $\gamma_a+\gamma_b$ is a positive root, which contradicts the minimality of $k$. Therefore, we must have $k=3$. However,
$$
  (\gamma_1+\gamma_2+\gamma_3,\beta) = (\alpha+2\beta,\beta) = 3
$$
and
$$
  (\gamma_1+\gamma_2+\gamma_3,\alpha+\beta) = (\alpha+2\beta,\alpha+\beta) = 3
$$
implies that $(\gamma_a,\beta) =  (\gamma_a,\alpha+\beta) =1$ for all $a \in \{1,2,3\}$. Therefore, for every $a \in \{1,2,3\}$, each of $\beta-\gamma_a$ and $\alpha+\beta-\gamma_a$ is a (positive or negative) root. Let $\height(\gamma)\in \mathbb{Z}_{>0}$ denote the height of a positive root $\gamma\in \Delta^+$. Because $\gamma_1+\gamma_2+\gamma_3=\alpha+2\beta$, we cannot have two or more of the $\gamma_a$'s of height $\geq \height(\alpha+\beta)$, and so either

\medskip

\begin{itemize}[leftmargin=*]

\item
$\height(\gamma_1),\height(\gamma_2),\height(\gamma_3)<\height(\alpha+\beta)$. Thus, we have that $\alpha+\beta-\gamma_a = \delta_a$ is a positive root for all $a \in \{1,2,3\}$, and we must have $\delta_a < \alpha$ for all $a$ due to the minimality of the pair $\alpha < \beta$ and the convexity. However, the fact that $\delta_1+\delta_2+\delta_3 = 2\alpha+\beta$ but $\delta_1,\delta_2,\delta_3 < \alpha < \alpha+\beta$ yields a contradiction.\footnote{A more general statement is true for convex orders $<$ on $\Delta^+$: one cannot have $\alpha_1 < \dots < \alpha_k < \beta_1 < \dots < \beta_l$ such that $\alpha_1+\dots+\alpha_k = \beta_1+\dots+\beta_l$.}

\medskip

\item
Exactly one of the $\gamma_a$'s has a height larger than $\height(\alpha+\beta)$. Thus, up to relabeling, we may assume that $\alpha+\beta-\gamma_1 = \delta_1$, $\alpha+\beta-\gamma_2 = \delta_2$ but $\delta_3 = \gamma_3 - \alpha - \beta$ for positive roots $\delta_1,\delta_2,\delta_3$. As before, we must have
$$
  \delta_1,\delta_2 < \alpha < \alpha+\beta < \gamma_1,\gamma_2,\gamma_3 < \beta
$$
and $\delta_3 > \gamma_3$ by convexity. However, as explained before, $\varepsilon = \gamma_3 - \beta = \delta_3 + \alpha$ is also a positive root. Furthermore, we have $\varepsilon > \alpha$ and $\varepsilon < \gamma_3 < \beta$ by convexity. But then the equality
$$
  \gamma_1+\gamma_2+\varepsilon = \gamma_1+\gamma_2+\gamma_3 - \beta = \alpha+\beta
$$
yields a contradiction to Claim \ref{claim:1}.

\end{itemize}

\medskip
\noindent
This completes our proof of Claim~\ref{claim:2}, and hence also of Lemma~\ref{lem:minimal ar}.
\end{proof}


\medskip

\subsection{Specialization maps for the AR partial order}
\label{sub:specialization ar}

Let $\tau\colon I \rightarrow \BZ$ be any function with the property that $\tau(i) = \tau(j)+1$ if there exists an arrow from $i$ to $j$ in the quiver $Q$. Such a map exists because finite type Dynkin diagrams do not have cycles, and it is unique up to a simultaneous translation. In Definition \ref{def:shuf} (and the paragraph following it), we noticed that elements of the shuffle algebra are defined as a certain Laurent polynomial $r$ divided by certain linear factors. More precisely, elements of the shuffle algebra $R \in \CS_\bv$ are of the form
\begin{equation}
\label{eqn:shuf ar}
  R(\dots,z_{i1},\dots,z_{iv_i},\dots) = \frac {r(\dots,z_{i1},\dots,z_{iv_i},\dots)}{\prod_{\oij} \prod_{1 \leq b \leq v_i}^{1\leq c \leq v_j} (z_{ib} - z_{jc})}
\end{equation}
where $r$ is a color-symmetric Laurent polynomial which satisfies the following three-variable wheel conditions whenever $i$ and $j$ are connected by an edge
\begin{equation}
\label{eqn:wheel ar}
  r(\dots,z_{ib},\dots) \Big|_{z_{i1} = qz_{j1}, z_{i2} = q^{-1}z_{j1}} = \, 0
\end{equation}
(indeed, \eqref{eqn:wheel ar} is just the particular case of \eqref{eqn:wheel} when $a_{ij} = -1)$.

\medskip

\begin{definition}
\label{def:specialization ar}
For any $\bv = (v_i)_{i \in I} \in \nn$, define the specialization map
\begin{equation}
\label{eqn:specialization ar}
  \CS_{\bv} \xrightarrow{\espec^{(x)}_{\bv}} \BQ(q)[x,x^{-1}],
\end{equation}
$$
  \espec^{(x)}_\bv(R) = \gamma_{\bv}^{(x)} \cdot r(\dots,z_{i1},\dots,z_{iv_i},\dots) \Big|_{z_{ib} \mapsto x q^{\tau(i)}, \forall i,b}
$$
where $r$ is the Laurent polynomial associated to $R \in \CS_{\bv}$ by formula \eqref{eqn:shuf ar} \footnote{Compared with our general expectation in \eqref{eqn:collection}, the specialization map \eqref{eqn:specialization ar} sets all the variables $z_{ib}$ for a given $i\in I$  to one and the same power of $q$ (times $x$).} and
\begin{equation}
\label{eqn:formula gamma}
  \gamma_{\bv}^{(x)} = q^{-\sum_{\oij} \tau(i) v_iv_j } \left[ x q^{-\frac 12} (q-q^{-1}) \right]^{-\bv \cdot \bv}
\end{equation}
\end{definition}

\noindent
Because of $\gamma_{\bv}^{(x)}$, the map \eqref{eqn:specialization ar} actually takes values in $\BQ(q^{\frac 12})[x,x^{-1}]$, but we will ignore this technicality, as it will produce no meaningful effects in the present paper.


\medskip

\subsection{A key result}

In the present Subsection, we will prove some key results pertaining to the specialization maps \eqref{eqn:specialization ar}. Let $\bv \cdot \bw = \sum_{i \in I} v_i w_i$ for all $\bv,\bw \in \nn$.

\medskip

\begin{lemma}
\label{lem:key}
For any $\bv,\bw \in \nn$ and any $R \in \CS_{\bv+\bw}$, we have
\begin{equation}
\label{eqn:count}
\begin{split}
  & \frac {\espec_{\bv}^{(x)} \otimes \espec_{\bw}^{(y)}(R)}{\prod_{i,j \in I} \zeta_{ij} \left( \frac {xq^{\tau(i)}}{yq^{\tau(j)}} \right)^{v_iw_j}} = \\
  & \quad \quad \frac{\gamma_{\bv}^{(x)} \gamma_{\bw}^{(y)}  \cdot r(\dots, xq^{\tau(i)},\dots,yq^{\tau(j)},\dots)}
       {(x-y)^{-\langle \bv,\bw \rangle} (y q^2 - x)^{\bv \cdot \bw - \langle \bw,\bv \rangle} (x - y q^{-2})^{\bv \cdot \bw}q^{\sum_{\oij}(\tau(i)v_iw_j+\tau(j)v_jw_i)}}
\end{split}
\end{equation}
In the formula above, $\espec_{\bv}^{(x)} \otimes \espec_{\bw}^{(y)}(R)$ means that we split the variables of $R$ into two sets, to which we separately apply the specialization maps $\espec_{\bv}^{(x)}$ and $\espec_{\bw}^{(y)}$.
\end{lemma}

\medskip

\begin{proof}
By the definition of the specialization maps in \eqref{eqn:specialization ar}, we have
$$
  \spec_{\bv}^{(x)} \otimes \spec_{\bw}^{(y)}(R) =
  \frac{\gamma_{\bv}^{(x)} \gamma_{\bw}^{(y)} \cdot r(\dots, xq^{\tau(i)},\dots,yq^{\tau(j)},\dots)}
       {\prod_{\oij} \left[ (x - y q^{-1})^{v_i w_j}(yq - x)^{w_i v_j} q^{\tau(i)v_iw_j+\tau(j)v_jw_i}\right]}
$$
Meanwhile, by the very definition of $\zeta_{ij}$ in~\eqref{eqn:zeta}, we have
$$
  \zeta_{ij} \left( \frac {x q^{\tau(i)}}{yq^{\tau(j)}} \right) =
  \begin{cases}
    \frac {x-yq^{-2}}{x-y} &\text{if } i = j \\
    \frac {x-y}{x-yq^{-1}} &\text{if there is an arrow }\oij \\
    \frac {yq^2-x}{yq-x} &\text{if there is an arrow }\oji \\
    1 &\text{otherwise}
  \end{cases}
$$
Dividing the formulas above yields \eqref{eqn:count}.
\end{proof}

\medskip
\noindent
Because of the wheel conditions \eqref{eqn:wheel ar}, the numerator of the RHS of \eqref{eqn:count} is also divisible by a number of linear factors of the form $(x-yq^{\pm 2})$. In the following result, we count these linear factors in the case when $\bv$ and $\bw$ are positive roots.

\medskip

\begin{proposition}
\label{lem:pole ar}
If $\bv$ and $\bw$ are positive roots, then for any $R \in \CS_{\bv+\bw}$, the rational function \eqref{eqn:count} has a pole at $x=yq^2$ of order
$$
  \leq \max \big( 0,-\langle \bw,\bv\rangle \big)
$$
and a pole at $x=yq^{-2}$ of order
$$
  \leq \max \big( 0,\langle \bv,\bw\rangle \big)
$$
If $\bv < \bw$, then the inequalities $\langle \bv,\bw \rangle \leq 0 \leq \langle \bw,\bv \rangle$ imply that the maxima above are both 0, hence, the rational function \eqref{eqn:count} does not have poles at $x = y q^{\pm 2}$.
\end{proposition}

\medskip

\begin{proof}
As a consequence of Theorem \ref{thm:iso}, any element of the shuffle algebra $\CS_{\bv+\bw}$ is a linear combination of expressions of the form
\begin{equation}
\label{eqn:shuffle product}
  R = z_{i_11}^{d_1} * \ldots * z_{i_k1}^{d_k}
\end{equation}
for any sequences $i_1,\dots,i_k \in I$, $d_1,\dots,d_k \in \BZ$ such that $\bs^{i_1} + \dots + \bs^{i_k} = \bv + \bw$. It therefore suffices to prove the required claims for $R$ as in \eqref{eqn:shuffle product}. In calculating the LHS of \eqref{eqn:count}, the variables of $R$ will be specialized to the multiset
\begin{equation}
\label{eqn:M}
  M = \Big\{\dots,\underbrace{xq^{\tau(i)},\dots,xq^{\tau(i)}}_{v_i \text{ occurrences}},\dots,\underbrace{yq^{\tau(j)},\dots,yq^{\tau(j)}}_{w_j \text{ occurrences}},\dots \Big\}
\end{equation}
By the very definition of the shuffle product, for $R$ as in \eqref{eqn:shuffle product} we have
\begin{equation}
\label{eqn:explicit}
  \text{LHS of \eqref{eqn:count}} \ =
  \sum_{\text{certain total orders }\succ\text{ on }M} \text{monomial} \prod_{x q^{\tau(i)} \succ y q^{\tau(j)}}
  \frac {\zeta_{ji} \left( \frac {yq^{\tau(j)}}{xq^{\tau(i)}} \right)}{\zeta_{ij} \left( \frac {xq^{\tau(i)}}{yq^{\tau(j)}} \right)}
\end{equation}
Since
\begin{equation}
\label{eqn:ratio of zetas}
  \frac {\zeta_{ji} \left( \frac {yq^{\tau(j)}}{xq^{\tau(i)}} \right)}{\zeta_{ij} \left( \frac {xq^{\tau(i)}}{yq^{\tau(j)}} \right)} =
  \begin{cases}
    \frac {x-yq^2}{xq^2-y} &\text{if }i = j \\
    \frac {xq-yq^{-1}}{x-y} &\text{if there is an arrow }\oij \\
    \frac {x-y}{xq^{-1}-yq} &\text{if there is an arrow }\oji \\
    1 &\text{otherwise}
  \end{cases}
\end{equation}
the orders of the poles at $x=yq^2$ and $x=yq^{-2}$ in any given summand of \eqref{eqn:explicit}~are
$$
  A = - \sum_{i \in I} \# \Big\{\big(xq^{\tau(i)} \succ y q^{\tau(i)}\big)\Big\} + \sum_{\overrightarrow{ji}} \# \Big\{\big(xq^{\tau(i)} \succ y q^{\tau(j)}\big)\Big\}
$$
$$
  B = \sum_{i \in I} \#\Big\{\big(xq^{\tau(i)} \succ y q^{\tau(i)}\big)\Big\} - \sum_{\overrightarrow{ij}} \# \Big\{\big(xq^{\tau(i)} \succ y q^{\tau(j)}\big)\Big\}
$$
respectively. It suffices to consider only those total orders $\succ$ on $M$ of \eqref{eqn:M} for which $x q^{\tau(i)} \succ x q^{\tau(j)}$ and $y q^{\tau(i)} \succ y q^{\tau(j)}$ for all arrows $\overrightarrow{ij}$ (indeed, if one of such inequalities failed, then the corresponding summand in the RHS of \eqref{eqn:explicit} would include the factor $\zeta_{ij}(q) = 0$ for adjacent $i,j$). We will refer to such total orders as ``allowable". Therefore, it just remains to show that
\begin{equation}
\label{eqn:need to show}
  A \leq \max \big( 0,-\langle \bw,\bv\rangle \big) \qquad \text{and} \qquad B \leq \max \big( 0,\langle \bv,\bw\rangle \big)
\end{equation}
for any allowable total order $\succ$ on $M$. We will only prove the inequality involving $A$, since it implies the inequality involving $B$ due to the identity
$$
  B = \langle \bv,\bw \rangle + (A \text{ for the opposite quiver and the opposite order})
$$

\medskip
\noindent
Fix an allowable total order $\succ$ on $M$, and we will call a variable $xq^{\tau(i)}$ (respectively $y q^{\tau(j)}$) ``distinguished" if it is greater (respectively smaller) than any variable $y q^{\tau(i)}$ (respectively $x q^{\tau(j)}$) with respect to the total order $\succ$. The total number of distinguished variables $xq^{\tau(i)}$ (respectively $yq^{\tau(j)}$) will be encoded by the degree vectors $\bv' \in [\b0,\bv]$ (respectively $\bw' \in [\b0,\bw]$), meaning $0\leq v'_i\leq v_i$ (respectively $0\leq w'_i\leq w_i$) for any $i\in I$. As the order $\succ$ is allowable, if we have two variables
$$
  xq^{\tau(i)} \succ yq^{\tau(j)}
$$
for some arrow $\oji$, then both variables $xq^{\tau(i)}$ and $yq^{\tau(j)}$ must be distinguished. We therefore conclude that
$$
  A \leq -\sum_{i\in I}(v_i'w_i + v_iw_i' - v_i'w_i') + \sum_{\oji} v_i' w'_j
$$
with the term $-(v_i'w_i + v_iw_i' - v_i'w_i')$ counting those pairs of variables $xq^{\tau(i)} \succ yq^{\tau(i)}$ where at least one of the variables is distinguished. Therefore, \eqref{eqn:need to show} follows from the following combinatorial result.

\medskip

\begin{claim}
If $\bv,\bw$ are positive roots, and $\bv' \in [\b0, \bv], \bw' \in [\b0,\bw]$ are arbitrary, then
\begin{equation}
\label{eqn:claim ineq}
  -\sum_{i\in I}(v_i'w_i + v_iw_i' - v_i'w_i') + \sum_{\oji} v_i' w'_j \leq \max \big( 0,-\langle \bw,\bv\rangle \big)
\end{equation}
\end{claim}

\medskip
\noindent
It remains to prove the Claim above, and we will prove it in the case when $\bv'$ and $\bw'$ are allowed to have any real coordinates in the boxes $[\b0,\bv]$ and $[\b0,\bw]$, respectively. Because the aforementioned boxes are compact, the LHS of \eqref{eqn:claim ineq} must reach its maximum at a point in the box $[\b0,\bv] \times [\b0,\bw]$. However, because the LHS of \eqref{eqn:claim ineq} is linear in each coordinate $v_i'$, $w_i'$ (and a linear function on an interval is maximized at one of the endpoints of the interval), it remains to prove \eqref{eqn:claim ineq} when each variable $v_i',w_i'$ is equal to either 0 or $v_i,w_i$, respectively. Thus, for any decompositions
$$
  I = I' \sqcup I'' = \tilde{I}' \sqcup \tilde{I}''
$$
it remains to prove \eqref{eqn:claim ineq} when $v_i' = v_i$ for $i \in I'$, $v_i' = 0$ for $i \in I''$, $w_i' = w_i$ for $i \in \tilde{I}'$, $w_i' = 0$ for $i \in \tilde{I}''$. Thus, we need to prove the following inequality
$$
  -\sum_{i \in I \setminus (I'' \cap \tilde{I}'')} v_iw_i \ + \sum_{\oji,\, i \in I',\, j \in \tilde{I}'} v_i w_j \, \leq \, \max \big( 0,-\langle \bw,\bv\rangle \big)
$$
Clearly, the left-hand side of the equation above is maximized when $\tilde{I}' = I'$, $\tilde{I}'' = I''$. Therefore, it remains to prove that
$$
  - \langle \bw_{I'} ,\bv_{I'} \rangle_{I'} \leq \max \big( 0,-\langle \bw,\bv\rangle \big)
$$
where $\bv_{I'}, \bw_{I'}$ denote the projections of the vectors $\bv,\bw \in \nn$ onto the coordinates indexed by $I'$, while $\langle \cdot , \cdot \rangle_{I'}$ denotes the restriction of \eqref{eqn:lattice pairing}  to $\BZ^{I'}\times \BZ^{I'} \subset \BZ^{I}\times \BZ^{I}$.

\medskip
\noindent
From the point of view of quiver representations, the inequality above reads
\begin{equation}
\label{eqn:quiver ineq}
  \dim \text{Ext}^1(W_{I'},V_{I'}) - \dim \text{Hom}(W_{I'},V_{I'}) \leq \dim \text{Ext}^1(W,V)
\end{equation}
for any indecomposable representations $V,W$ of the quiver $Q$, where $V_{I'},W_{I'}$ denote their restrictions to the full subquiver corresponding to the vertex set $I' \subset I$ (here we used~\eqref{eqn:Hom-or-Ext} again). In fact, inequality \eqref{eqn:quiver ineq} follows from the stronger inequality
\begin{equation}
\label{eqn:quiver ineq strong}
  \dim \text{Ext}^1(W_{I'},V_{I'}) \leq \dim \text{Ext}^1(W,V)
\end{equation}
that holds for all finite-dimensional $Q$-representations $V$ and $W$. In turn, inequality \eqref{eqn:quiver ineq strong} follows immediately from the claim below.

\medskip

\begin{claim}
\label{claim:gen}
For any finite-dimensional representations $V,W$ of the quiver $Q$, let $V_{I'},W_{I'}$ denote their restrictions to the full subquiver $Q'$ corresponding to the vertex set $I' \subset I$. Then, the natural restriction map
\begin{equation}
\label{eqn:surjection}
  \text{Ext}^1(W,V) \twoheadrightarrow \text{Ext}^1(W_{I'},V_{I'})
\end{equation}
is surjective.
\end{claim}

\medskip
\noindent
It thus remains to prove Claim \ref{claim:gen}. Any element of $\text{Ext}^1(W_{I'},V_{I'})$ can be represented by a collection of short exact sequences of vector spaces
\begin{equation}
\label{eqn:extension}
  \big\{ 0 \rightarrow V_i \rightarrow X_i \rightarrow W_i \rightarrow 0 \big\}_{i \in I'}
\end{equation}
which are compatible with the arrow maps of the subquiver $Q'$ with the vertex set $I'$. Let us choose $\mathbb{F}_{q^2}$ vector space splittings $X_i \simeq V_i \oplus W_i$ for all $i \in I'$, with respect to which the short exact sequences above are
$$
  0 \rightarrow V_i \xrightarrow{(\text{Id},0)} V_i \oplus W_i \xrightarrow{(0,\text{Id})} W_i \rightarrow 0, \qquad \forall\, i \in I'
$$
although we do not claim that the arrow maps in $Q'$ respect these decompositions. To show that our given extension lies in the image of the map \eqref{eqn:surjection}, we must extend it to an extension of the $Q$-representations $V$ and $W$. To this end, we consider the split short exact sequences
$$
  0 \rightarrow V_i \xrightarrow{(\text{Id},0)} V_i \oplus W_i \xrightarrow{(0,\text{Id})} W_i \rightarrow 0, \qquad \forall\, i \in I \setminus I'
$$
and extend $\{V_i \oplus W_i\}_{i \in I}$ to a $Q$-representation by defining the arrow maps to be
$$
  V_i \oplus W_i \xrightarrow{(\phi_{\overrightarrow{ij}}, \psi_{\overrightarrow{ij}})} V_j \oplus W_j
$$
whenever either $i$ or $j$ lie in $I \setminus I'$ (above, $\phi_{\overrightarrow{ij}}\colon V_i\to V_j$ and $\psi_{\overrightarrow{ij}}\colon W_i\to W_j$ denote the arrow maps in the quiver representations $V$ and $W$, respectively, see~\eqref{eqn:arrow map}), while using the arrow maps $X_i\to X_j$ whenever $i,j\in I'$. The resulting $Q$-representation yields an element of $\text{Ext}^1(W,V)$ which lifts our choice of extension~\eqref{eqn:extension}.
\end{proof}


\medskip

\subsection{Fused currents for the AR partial order}

In the present Subsection, we invoke Proposition \ref{prop:completion pairing} to define for any $(\alpha,d) \in \Delta^+ \times \BZ$
\begin{equation*}
  f_{\alpha,d} \in \wUUm_{-\alpha,d}
\end{equation*}
by the formula (we let $f_{\alpha}(x) = \sum_{d \in \BZ} \frac {f_{\alpha,d}}{x^d}$ and call them the ``fused currents")
$$
  \Big \langle R, f_{\alpha}(x) \Big \rangle = \spec^{(x)}_\alpha(R), \qquad \forall\, R \in \CS_{\alpha}
$$
Let us recall the fused currents $\tf_{\alpha}(x)$ of \cite{DK} (see Subsections \ref{sub:fused} and \ref{sub:specialization}) in the particular case of a type ADE Dynkin diagram, defined with respect to any total convex order of the positive roots that refines the Auslander--Reiten partial order.

\medskip

\begin{conjecture}
\label{conj:match}
For any orientation $Q$ of a type ADE Dynkin diagram and any total convex order of the positive roots that refines the AR partial order, we have
$$
  \tf_{\alpha}(x) = c_{\alpha}^{(x)} \cdot f_{\alpha}(x) \quad \textit{or equivalently} \quad \tespec_{\alpha}(x) = c_{\alpha}^{(x)} \cdot \espec_{\alpha}(x)
$$
for all $\alpha \in \Delta^+$, where $c_{\alpha}^{(x)} \in \BQ(q,x)^\times$ is some scalar prefactor.
\end{conjecture}

\medskip
\noindent
To motivate Conjecture \ref{conj:match}, let us carry out the program from Subsection \ref{sub:specialization} for the objects $f_\alpha(x)$ and $\spec_{\alpha}^{(x)}$ defined in the present Section.

\begin{proof}[Proof of Theorem \ref{thm:main intro}, specifically equation \eqref{eqn:fused intro equality}.]
Combining together Lemma \ref{lem:minimal ar} and Proposition \ref{lem:pole ar} implies that if $\alpha < \beta$ is a minimal pair such that $\alpha+\beta \in \Delta^+$, then
\begin{equation}
\label{eqn:simple}
  \frac {\spec_{\alpha}^{(x)} \otimes \spec_{\beta}^{(y)}(R)}{\prod_{i,j \in I} \zeta_{ij} \left( \frac {xq^{\tau(i)}}{yq^{\tau(j)}} \right)^{\alpha_i\beta_j}} =
  \frac {\text{Laurent polynomial}}{x-y}
\end{equation}
for any $R \in \CS_{\alpha+\beta}$. Meanwhile, the rational function \eqref{eqn:rational function} is easily computed to~be
$$
  \prod_{i,j \in I} \left[ \frac {\zeta_{ji} \left( \frac {yq^{\tau(j)}}{xq^{\tau(i)}} \right)}{\zeta_{ij} \left( \frac {xq^{\tau(i)}}{yq^{\tau(j)}} \right)} \right]^{\alpha_i\beta_j}
  \overset{\eqref{eqn:ratio of zetas}}{=}
  \left(\frac {x-y}{xq-yq^{-1}}\right)^{\langle \alpha,\beta \rangle} \left(\frac {xq^{-1}-yq}{x-y}\right)^{\langle \beta,\alpha \rangle}
  \overset{\eqref{eqn:minimal ar}}{=} \frac {xq-yq^{-1}}{x-y}
$$
Thus, the difference of \eqref{eqn:expansion 1} and \eqref{eqn:expansion 2}, specifically
$$
  \Big \langle R, f_{\alpha}(x) f_{\beta}(y) \Big \rangle \Big|_{|x| \ll |y|} - \left \langle R,  f_{\beta}(y)f_{\alpha}(x)  \frac {xq-yq^{-1}}{x-y} \right \rangle \Big|_{|x| \gg |y|}
$$
is just the difference between the expansions at $|x| \ll |y|$ and $|x| \gg |y|$ of the rational function \eqref{eqn:simple}. Explicitly, the rational function in question is \eqref{eqn:count} for $\bv = \alpha$ and $\bw = \beta$. The aforementioned rational function has a single simple pole at $x = y$ with residue given by
$$
  \frac{\gamma_{\alpha}^{(x)} \gamma_{\beta}^{(x)}  \cdot r(\dots, xq^{\tau(i)},\dots)}
       {(x q^2 - x)^{\alpha \cdot \beta - \langle \beta,\alpha \rangle} (x -x q^{-2})^{\alpha \cdot \beta}q^{\sum_{\oij}(\tau(i)\alpha_i\beta_j+\tau(j)\alpha_j\beta_i)}} =
  \gamma_{\alpha+\beta}^{(x)} \cdot  r(\dots, xq^{\tau(i)},\dots)
$$
(the latter equality is the reason for the specific formula of $\gamma_{\bv}^{(x)}$ in \eqref{eqn:formula gamma}; it is quite elementary, based on $\langle \beta,\alpha\rangle = 0$ of Lemma~\ref{lem:minimal ar} and $\tau(i)-\tau(j)=1$ for any arrow $\oij$ in $Q$, so we leave its proof as an exercise to the reader). Once we observe that the residue above is precisely $\spec_{\alpha+\beta}^{(x)}(R)$, we conclude that
\begin{equation}
\label{eqn:fused ar comm}
  f_{\alpha}(x) f_{\beta}(y) -  f_{\beta}(y)f_{\alpha}(x)  \frac {xq-yq^{-1}}{x-y} = \delta\left(\frac xy\right) f_{\alpha+\beta}(x)
\end{equation}
\end{proof}

\noindent
We remark that formula \eqref{eqn:fused ar comm} allows one to inductively define the fused currents corresponding to the AR partial order from $f_{\alpha_i}(x) = f_i(x)$ for simple roots $\{\alpha_i\}_{i \in I}$.


\end{document}